%% file: main.tex
\documentclass[reqno]{amsart}
\usepackage[foot]{amsaddr}

\usepackage{graphicx}
\usepackage{subfigure}
\usepackage{amsmath}
\usepackage{bm}
\usepackage{tikzit}
\usepackage{bbm}
\usetikzlibrary{math}
\usepackage{url}
\usepackage{amssymb}

\newtheorem{theorem}{Theorem}[section]
\newtheorem{remark}[theorem]{Remark}
\newtheorem{lemma}[theorem]{Lemma}

\newtheorem{proposition}[theorem]{Proposition}

\newtheorem{example}[theorem]{Example}

\newtheorem{corollary}[theorem]{Corollary}

\newtheorem{definition}[theorem]{Definition}

\DeclareMathOperator*{\Prob}{\mathbb{P}}
\newcommand{\prb}[1]{\mathbb{P}\left[ #1 \right]}

\newcommand{\compnorm}{\mathcal{N}_\mathcal{C}(0, 1)}
\newcommand{\compnormvec}{\mathcal{N}_\mathcal{C}(0, I_n)}

\title{Estimating the Numerical Range with a Krylov Subspace}
\author{Cecilia Chen}
\address{Department of Mathematics, Massachusetts Institute of Technology, \newline \indent Cambridge, MA, 02139 USA.}
\email{cdchen@mit.edu}
\author{John Urschel}
\email{urschel@mit.edu}
\begin{document}

\begin{abstract}
Krylov subspace methods are a powerful tool for efficiently solving high-dimensional linear algebra problems. In this work, we study the approximation quality that a Krylov subspace provides for estimating the numerical range of a matrix. In contrast to prior results, which often depend on the gaps between eigenvalues, our estimates depend only on the dimensions of the matrix and Krylov subspace, and the conditioning of the eigenbasis of the matrix. In addition, we provide nearly matching lower bounds for our estimates, illustrating the tightness of our arguments.
\end{abstract}

\maketitle

\section{Introduction}

Moment-based methods are ubiquitous in applied mathematics, and numerical linear algebra is no exception. Krylov subspace methods are an incredibly popular family of algorithms that approximate the solution to some problem involving a matrix $A$ using the Krylov subspace
\[\mathcal{K}_m(A, \bm b) = \mathrm{span}\{ \bm b, A \bm b, A^2 \bm b,\ldots ,A^{m-1} \bm b\}.\]
This includes methods for approximating linear systems (conjugate gradient method \cite{hestenes1952methods}, GMRES \cite{saad1986gmres}, MINRES \cite{paige1975solution}, hybrid methods \cite{hybridarnoldi1993} etc.), extremal eigenvalue problems (Arnoldi iteration \cite{arnoldi1951principle}, Lanczos method \cite{lanczos1950iteration}, etc.), pseudospectra \cite{pseudo}, and matrix functions (e.g., estimating $f(A) \bm{b}$ or $\bm{b}^* f(A) \bm{b}$). In many applications, $m$ is much smaller than $n$, and a key benefit of these methods is that they only involve matrix-vector products, allowing for fast computation (sometimes without even forming $A$ explicitly). Here we focus on the quality of the estimate on the numerical range 
\[ W(A) = \left\{  \frac{\bm x^* A \bm x}{\bm x^* \bm x} \, \bigg| \, \bm x \in \mathbb{C}^{n \times n} \right\} \]
of a matrix $A$ produced using the numerical range of a projection onto a Krylov subspace. In particular, we consider the orthogonal projection of $A$ onto the Krylov subspace $\mathcal{K}_m(A, \bm b)$ for some vector $\bm b$, denoted by $H_m$. Such estimates are important not only in the computation of extremal eigenvalues, but also for error estimates of other methods. For instance, standard error bounds for the residual in the GMRES algorithm for $A \bm{x} = \bm{b}$ after $m$ steps depend on the quantity $\min_{p \in \mathcal{P}_m \, \mathrm{s.t.}\,p(0)=1} \max_{\lambda \in \Lambda(A)} |p(\lambda)|$, where $\mathcal{P}_m$ is the set of complex polynomials of degree at most $m$ and $\Lambda(A)$ is the spectrum of $A$ \cite[Proposition 6.32]{saad2003iterative}. If $W(H_m)$ provides a good estimate of $W(A)$, or even $\mathrm{conv}(\Lambda(A))$, then computing $W(H_m)$ and estimating the separation from zero can produce guarantees for the convergence rate. For more details about Krylov subspace methods in general, we refer the reader to \cite{saad2003iterative,saad2011numerical} (see also \cite{saad2022origin} for a historical perspective). Here, we consider only the mathematical properties of a Krylov subspace and neglect matters of implementation and rounding errors in floating point arithmetic (for a discussion of these aspects, see \cite{greenbaum1997iterative,liesen2013krylov,meurant2006lanczos,meurant2006lanczos2}).

Typically, estimates for approximating an extreme eigenvalue $\lambda$ with a Krylov subspace $\mathcal{K}_{m+1}(A, \bm b)$ depend on the quality of the initial guess $\bm b$ in relation to the eigenspace of $\lambda$, the eigenvector condition number, and the quantity $\min_{p \in \mathcal{P}_m \, \mathrm{s.t.}\,p(\lambda)=1} \max_{\mu \in \Lambda(A)\backslash \lambda} |p(\mu)|$. The latter quantity is small when there is separation between $\lambda$ and $\Lambda(A) \backslash \lambda$ in the complex plane, but can become arbitrarily close to one in many cases, producing unnecessarily pessimistic bounds. These issues can persist even when $A$ is Hermitian. For example, consider the tridiagonal matrix $A \in \mathbb{R}^{n \times n}$ resulting from the discretization of the Laplacian operator on an interval with Dirichlet boundary conditions (i.e., $A_{ii} = 2$, $A_{i,i+1} = A_{i+1,i} = -1$). This matrix has no small eigenvalue gaps, and so standard gap-dependent error estimates (e.g., \cite[Theorem 6.4]{saad2003iterative}) significantly over estimate the error in approximation for extreme eigenvalues, producing error bounds greater than one even when the Krylov subspace dimension equals $n$ (see \cite[Example 1.1]{urschel2021uniform} for details).

In practice, Krylov subspace methods perform well at estimating extreme eigenvalues, even when eigenvalue gaps are small. In \cite{musco2015randomized}, Musco and Musco argue that while a good estimate of an eigenvector naturally implies a good estimate of its corresponding eigenvalue, approximating the eigenvector may be unnecessary. Consider two close eigenvalues $\lambda$ and $\lambda'$ with corresponding eigenvectors $\bm{\varphi},\bm{\varphi}'$. It would be difficult to distinguish the components of $\bm{\varphi}$ and $\bm{\varphi}'$, but, for eigenvalues, an estimate near $\lambda$ is clearly near $\lambda'$ as well. Kuczynski and Wozniakowski were arguably the first to fully recognize this phenomenon in a quantitative way, producing a probabilistic upper bound of order $\ln^2 n /m^2$ on the quantity $(\lambda_{\max} (A) - \lambda^{(m)}_{\max})/(\lambda_{\max}(A) - \lambda_{\min}(A))$ for a real symmetric matrix $A \in \mathbb{R}^{n \times n}$, where $\lambda^{(m)}_{\max}$ is the largest eigenvalue of $H_m$, the orthogonal projection of $A$ onto $\mathcal{K}_{m}(A,\bm b)$, for an initial guess $\bm b$ sampled uniformly from the hypersphere \cite{kuczynski1992estimating}. This concept was later extended to the block setting. Musco and Musco introduced a block Krylov subspace method (a variant of the traditional block Lanczos method) that produces an $\epsilon$ approximation to the best rank $k$ approximation of $A$ using block size $k$ and order $\ln n / \sqrt{\epsilon}$ iterations (i.e., $\epsilon \approx \ln^2 n / m^2$, where $m$ is the number of iterations) \cite{musco2015randomized}. The original upper bound of order $\ln^2 n /m^2$ of Kuczynski and Wozniakowski has been shown to be essentially tight in a variety of settings, see \cite{garza2020lanczos,simchowitz2018tight,urschel2021uniform}. 

\begin{figure}
\begin{center}
\subfigure[$\Lambda(A)$, $\partial W(H_m)$, and $\Lambda(H_m)$ for normal $A$ with $\Lambda(A) = \{e^{2 \pi ij/n}\}_{j=1}^{n}$, $n = 10^4$, $\bm{b} \sim \mathrm{Unif}(\mathbb{S}^{n-1})$, and $m = 30$]{\includegraphics[height = 2.2in]{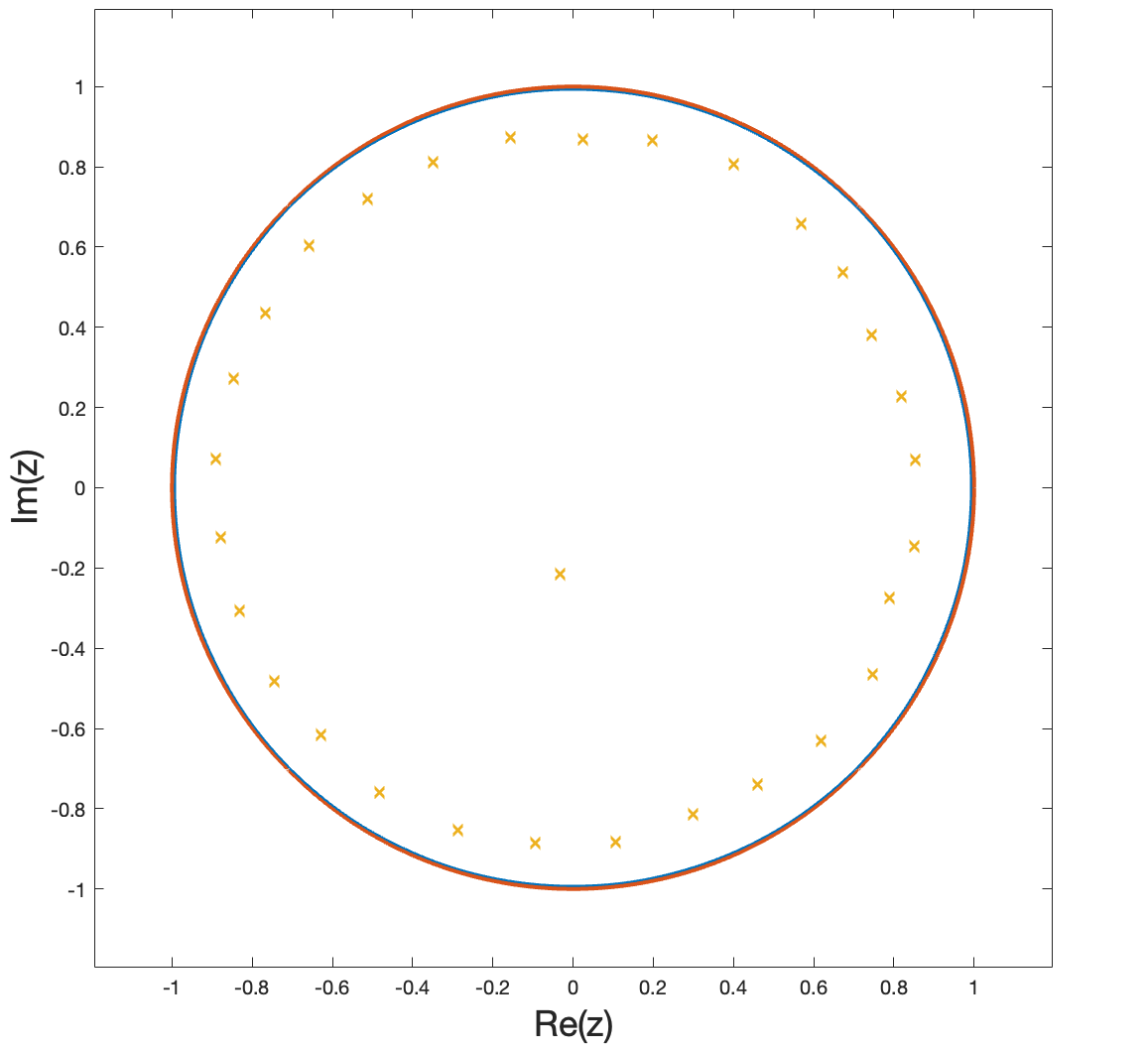}} \qquad
\subfigure[$d_H(\mathrm{conv}(\Lambda(A)),\mathrm{conv}(\Lambda(H_m))) \times m$ compared to $d_H(\mathrm{conv}(\Lambda(A)),W(H_m)) \times m^2$ for $m = 1,\ldots,30$]{\includegraphics[height = 2.2in]{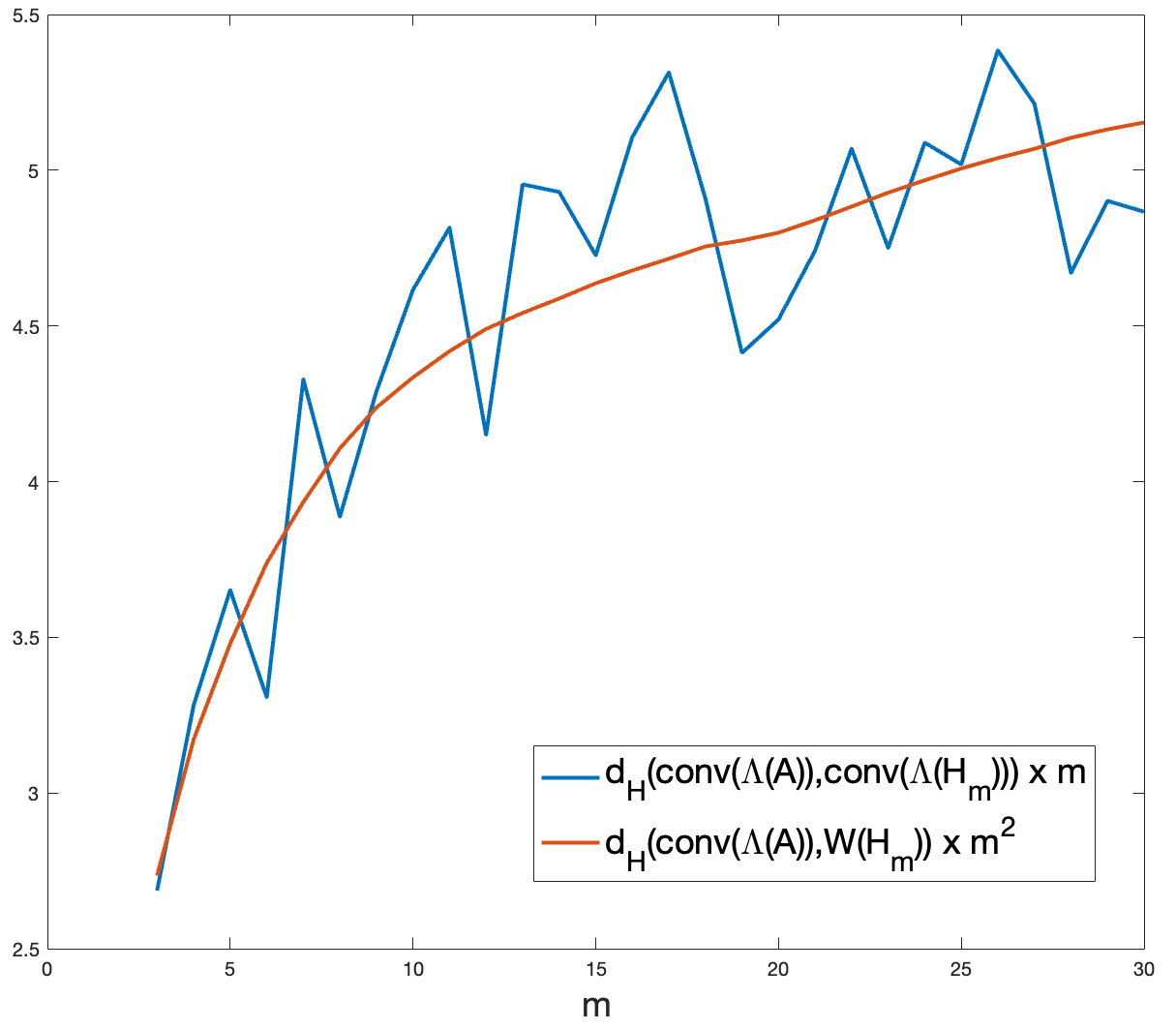}}
\caption{Estimating the eigenvalues of $A = \mathrm{diag}(\{e^{2 \pi i j/n}\}_{j=1}^n)$ using $\mathcal{K}_m(A,\bm b)$ for $\bm{b} \sim \mathrm{Unif}(\mathbb{S}^{n-1})$, where $n = 10^4$ and $m = 30$. In Subfigure (a), note that, while there is a decent gap between $\Lambda(A)$ and $\Lambda(H_m)$, the difference between $\Lambda(A)$ and $\partial W(H_m)$ is imperceptible. This is further quantified in Subfigure (b), where we note that the Hausdorff distance between $\mathrm{conv}(\Lambda(A))$ and $\mathrm{conv}(\Lambda(H_m))$ scales roughly like $1/m$, while the Hausdorff distance between $\mathrm{conv}(\Lambda(A))$ and $W(H_m)$ scales like $1/m^2$.}
\label{fig:ritzvsrange}
\end{center}
\end{figure}

However, to date, much less is known for the general, non-Hermitian setting. For instance, no such gap-independent upper or lower bounds have been produced for the extreme eigenvalues of non-Hermitian matrices. The situation here is more subtle than the Hermitian case. For instance, even if $A$ is normal, $H_m$ is usually non-normal (see \cite[Theorem 1]{huckle1994arnoldi}). This provides an additional difficulty for approximating the extreme eigenvalues of $A$ using the eigenvalues of $H_m$ (i.e., the Ritz values). In fact, in some cases, the numerical range $W(H_m)$ provides a far superior estimate to $\Lambda(A)$ than the Ritz values $\Lambda(H_m)$ themselves. See Figure \ref{fig:ritzvsrange} for an illustration of this phenomenon for a normal matrix with eigenvalues located at the $n^{th}$ roots of unity. For this reason, in this work we focus exclusively on the approximation provided by the numerical range $W(H_m)$ rather than the Ritz values $\Lambda(H_m)$. When $A$ is normal and $W(H_m)$ approximates $W(A)$ reasonably well, the profile of $W(H_m)$ inherently encodes information about the location of extreme eigenvalues of $A$, as they are the vertices of the polytope $W(A)$. When $A$ is non-normal, this precise recognition of extreme eigenvalues using $W(H_m)$ is no longer possible, yet still something can be said about the possible location of an eigenvalue $\lambda \in \Lambda(A)$ by bounding the distance between either $\mathrm{conv}(\Lambda(A))$ and $W(H_m)$ or $W(A)$ and $W(H_m)$.

\subsection{Our Contributions} Here we prove a number new results regarding the approximation quality of $W(H_m)$ to $W(A)$, independent of the distribution of $\Lambda(A)$, i.e., gap-free/uniform bounds. These are the first estimates of their type in the non-Hermitian setting. We break our contributions into three cases (all of the below results hold with high probability):
\begin{enumerate}
\item {\bf Normal Matrix $A$.} We prove that the Hausdorff distance between the numerical ranges of $A$ and $H_m$ is at most order $\ln n /m$ (Theorem \ref{thm:main}) and prove a lower bound of order $1/m$ for the distance (standard distance, not Hausdorff) between their boundaries for some family of matrices (Theorem \ref{thm:norm_lower}).
\item {\bf Normal Matrix $A$ with $\Lambda(A) \subset \mathbb{S}^0$.} When the spectrum of $A$ lies on the unit circle, we can further strengthen our estimates. In particular, we prove that the Hausdorff distance between the numerical ranges of $A$ and $H_m$ is at most order $\ln^2 n /m^2$ (Theorem \ref{thm:circle_upper}) and prove a lower bound of order $1/m^2$ for the distance between their boundaries for some family of matrices (Theorem \ref{thm:circle_lower}). The choice of $\mathbb{S}^0$ as the containing set is solely due to the popularity of unitary matrices. Similar bounds almost certainly hold for the boundary of any convex body.
\item {\bf Non-Normal Matrix $A$.} When $A$ is non-normal, the above estimates decay rapidly with eigenvector condition number (Example \ref{ex:bad}). We prove that the Hausdorff distance between the numerical ranges of $A$ and $H_m$ is at most order $\kappa^2 \ln n /m + \kappa (\kappa -1)$, where $\kappa$ is the condition number of a matrix $V$ with orthonormal columns that diagonalizes $A$ (Theorem \ref{thm:non-normal}). The decay of this estimate with $\kappa$ is consistent with observed worst-case behavior (see Example \ref{ex:bad}). However, when the matrices $V$ and $(V^{-1})^*$ each have nearly orthonormal columns (roughly, inner products smaller than $n^{-(1/2+\beta/2)}$, see Definition \ref{def:betanormal} for details), we can prove estimates for the convex hull of $\Lambda(A)$, even when $A$ is moderately non-normal. In particular, if $A$ satisfies the aforementioned condition for some $\beta>0$, then the maximum distance between a point $z \in \mathrm{conv}(\Lambda(A))$ and the numerical range of $H_m$ is at most order $\ln n /m$ when $m \ll n^{\beta/2}$ (Theorem \ref{thm:eigenhull}).
\end{enumerate}

The remainder of the paper is as follows. In Section \ref{sec:tech}, we introduce the notation, proof ideas, and technical tools that we use in this work. In Section \ref{sec:normal}, we consider the case when $A$ is normal, and when $A$ is normal with its spectrum on the unit circle $\Lambda(A) \subset \mathbb{S}^0$. In Section \ref{sec:nonnormal}, we treat the case of non-normal matrices.

\section{Technical Tools: Convexity, Polynomials, and Probability}\label{sec:tech}

Here, we detail the basic notation used throughout the paper, outline the proof techniques for our main results, and recall and prove a number of technical results regarding convex geometry, extremal polynomials, and probabilistic inequalities.

\subsection{Notation}\label{sub:notation}

Recall that the Krylov subspace for a matrix $A \in \mathbb{C}^{n \times n}$ and vector $\bm{b} \in \mathbb{C}^n$ of order $m$ is given by $\mathcal{K}_m(A,\bm{b}) = \{ p(A) \bm{b} \, | \, p \in \mathcal{P}_{m-1}\}$, where $\mathcal{P}_{m-1}$ is the set of complex polynomials of degree at most $m-1$. If $\dim(\mathcal{K}_m(A,\bm{b}))=m$, then we may define $m$ orthonormal vectors $\{\bm{q}^{(j)}\}_{j=1}^m$ with the property that
$\mathrm{span}\{\bm{q}^{(1)},\ldots,\bm{q}^{(j)}\} = \mathrm{span}\{\bm{b},A \bm{b},\ldots,A^{j-1} \bm{b}\}$ for $j = 1,\ldots,m$. If this is the case, we define $Q_m = [ \bm q^{(1)} \ldots \bm q^{(m)} ] \in \mathbb{C}^{n \times m}$, and denote by $H_m = Q_m^* A Q_m$ the orthogonal projection of $A$ onto $\mathcal{K}_m(A,\bm{b})$ in the basis of $Q_m$ (note $H_m$ is an upper Hessenberg matrix). For simplicity, if $\dim(\mathcal{K}_m(A,\bm{b}))=k<m$, we define $H_m=H_k$. Also recall that $W(A)= \big\{ \frac{\bm{x}^*A\bm{x}}{\bm{x}^* \bm{x}} \, \big| \, \bm{x} \in \mathbb{C}^n \big\}$ is the numerical range of a matrix. 

Let $\mathbb{S}^{n-1}$ denote the unit complex $n$-sphere in $\mathbb{C}^n$. Let $\compnorm$ be the standard complex normal distribution (i.e., for $z \sim \compnorm$,  $\mathrm{Re}(z),\mathrm{Im}(z) \sim \mathcal{N}(0,\tfrac{1}{2})$ and are independent) and $\compnormvec$ be the distribution of standard complex normal vectors of length $n$. Given that the Krylov subspace $\mathcal{K}_m(A,\bm{b})$ is invariant under scaling $\bm{b}$, we will often move between taking $\bm{b} \sim \mathrm{Unif}(\mathbb{S}^{n-1})$ or $\bm{b} \sim \compnormvec$, depending on which formulation better suits our needs. We say $X \overset{d}{=} Y$ if the random variables $X$ and $Y$ are equal in distribution.

To measure the proximity of two sets $S,T \subset \mathbb{C}$, let \begin{align*}
d(S,T) &= \inf_{\substack{s \in S,\\t \in T}} |s-t|, \\
\tilde{d}_H(S,T) &= \sup_{s \in S} d(s,T), \\
d_H(S,T) &= \max \{ \sup_{s \in S} d(s,T) , \sup_{t \in T} d(S,t) \}
\end{align*} denote the distance, one-sided Hausdorff distance, and Hausdorff distance. Note that, despite the name, $\tilde{d}_H(\cdot,\cdot)$ is not a distance. Furthermore for a set $S\subset \mathbb{C}$, let $\partial S$ be its boundary.

Finally, we recall miscellaneous and more basic notation that we will make use of. Let $T_m(z)$ denote the $m^{th}$ degree Chebyshev polynomial of the first kind, defined by $T_m(\cos \theta) = \cos m \theta$. Let $o_n(f(n))$ denote an arbitrary function that, when divided by $f(n)$, tends to zero as $n \rightarrow \infty$ and other parameters stay fixed. Let $\|\cdot\|_2$ be the vector and matrix $2$-norm, $\|\cdots\|_F$ be the Frobenius norm, and $\kappa(A) = \|A\|_2 \|A^{-1}\|_2$ be the matrix condition number,
and $\kappa_V(A) = \inf_{A = V \Lambda V^{-1}} \kappa(V)$ be the eigenvector condition number of $A$, that is, the smallest condition number over all matrices $V$ that diagonalize $A$. Let $i = \sqrt{-1}$, $\langle \cdot, \cdot\rangle$ be the complex inner product, $\mathrm{diam}(\cdot)$ be the diameter of a set, $\mathrm{conv}(\cdot)$ be the convex hull of a set, $\bm{e}_j$ be the $j^{th}$ standard basis vector, $\otimes$ be the Kronecker product, $\mathrm{diag}(\bm{x})$ be the diagonal matrix with $\bm{x}$ on its diagonal, and $[n] = \{1,\ldots,n\}$.

\subsection{Main Ideas and Eigenvalue Approximation as a Polynomial Problem}
In this work, we are primarily concerned with estimating $d_H\big(W(H_m),W(A)\big)$ for $\bm{b} \sim \mathrm{Unif}(\mathbb{S}^{n-1})$ uniformly sampled from the hypersphere. (See Subsection \ref{sub:convexity} for a pair of useful propositions concerning Hausdorff distances between convex bodies.) 

First, we sketch the rough argument for the upper bounds produced in this work (Theorems \ref{thm:main}, \ref{thm:circle_upper}, and \ref{thm:non-normal}). It is useful to think of the approximation to any particular eigenvalue $\lambda$ as a polynomial problem. Let $A \in \mathbb{C}^{n \times n}$ be a diagonalizable matrix with eigendecomposition $A = V \Lambda V^{-1}$, where the columns of $V$ have norm one, and $\bm{\alpha} = V^{-1} \bm{b}$. Then the distance $d(\lambda, W(H_m))$ equals
\begin{equation}\label{eqn:nonnormal_RQ}
\min_{\bm{x} \in \mathcal{K}_m(A,\bm{b})} \bigg| \lambda - \frac{\bm{x}^* A \bm{x}}{\bm{x}^* \bm{x}} \bigg| = \min_{p \in \mathcal{P}_{m-1}} \bigg| \lambda - \frac{[p(A) \bm{b}]^* A [p(A) \bm{b}]}{[p(A) \bm{b}]^* [p(A) \bm{b}]} \bigg| = \min_{p \in \mathcal{P}_{m-1}} \left| \frac{\bm{\alpha}^* p(\Lambda)^* V^*V(\lambda I - \Lambda) p(\Lambda) \bm{\alpha} }{\bm{\alpha}^* p(\Lambda)^* V^*Vp(\Lambda) \bm{\alpha} } \right|.
\end{equation}
In the special case where $A$ is normal, $V^* V = I$ and this expression can be simplified further
\begin{equation}\label{eqn:normal_RQ}
    \min_{\bm x \in \mathcal{K}_m(A,\bm b)} \bigg| \lambda - \frac{\bm x^* A \bm x}{\bm x^* \bm x} \bigg| = \min_{p \in \mathcal{P}_{m-1}} \left| \frac{\sum_{j=1}^n |\bm{\alpha}_j|^2  |p(\Lambda_j)|^2 (\lambda - \Lambda_j)}{\sum_{j=1}^n |\bm{\alpha}_j|^2  |p(\Lambda_j)|^2} \right|.
\end{equation}
To make this quantity small for some fixed $\bm b$, we construct an extremal polynomial $p \in \mathcal{P}_{m-1}$ in Subsection \ref{sub:poly} that is large at $\lambda$ and very small for all other eigenvalues at least some distance away from $\lambda$. In Subsection \ref{sub:probab}, we state some standard inequalities in high-dimensional probability, which allow us to handle the randomness of $\bm{b}$ and bound $d_H\big(W(H_m),W(A)\big)$ with high probability. In particular, we bound the probability that $\bm{b}$ is too close to orthogonal to any eigenspace. When $A$ is non-normal, the $V^*V$ and $\alpha = V^{-1}\bm b$ terms in Equation \ref{eqn:nonnormal_RQ} can produce barriers to this argument. To produce some estimates in the setting where $A$ is moderately non-normal, we show that, under certain conditions, the $V^*V$ can be removed to match the normal case with small additive and multiplicative error and that the distribution of $V^{-1}\bm b$ is close to as well-behaved as $\bm b$. 

Finally, the main idea behind the lower bounds produced here (Theorems \ref{thm:norm_lower} and \ref{thm:circle_lower}) consists of carefully choosing matrices $A$ for which $H_m$ is a perturbation (with respect to the randomness of $\bm{b}$) of a simple matrix that is easily analyzed, and measuring the distance between the numerical ranges of $A$ and $H_m$ by measuring their distances to this simpler matrix.

\subsection{Convex Geometry in the Plane}\label{sub:convexity} In order to quantify the approximation that $W(H_m)$ provides to either $W(A)$ or $\mathrm{conv}(\Lambda(A))$, we recall and prove two results regarding the approximation of a convex body in the complex plane.

\begin{proposition}[{\cite[Section 4]{polytope},\cite{popov}}]\label{prop:samplepolytope}
    Let $\mathcal{U}$ be a convex body in $\mathbb{C}$ with boundary length $L$. There exist $n$ points $S$ on the boundary of $\mathcal{U}$ such that the Hausdorff distance $$d_H(\mathcal{U}, \mathrm{conv}(S)) \leq \frac{L}{2n}\tan \frac{\pi}{n}.$$
\end{proposition}

\begin{proposition}\label{prop:polytopevertices}
    Let $\mathcal{U} \subset \mathbb{C}$ be a convex polytope with vertices $u_1,\ldots,u_n \in \mathbb{C}$, and $\mathcal{V} \subset \mathbb{C}$ be 
    a convex body with $d(u_j,\mathcal{V}) \le \delta$ for all $j \in [n]$. Then $\tilde d_H(\mathcal{U},\mathcal{V}) \le \delta$. Furthermore, if $\mathcal{V} \subset \mathcal{U}$, then $d_H(\mathcal{U}, \mathcal{V}) \leq \delta$.
\end{proposition}

\begin{proof}
    Take any point $u \in \mathcal{U}$. Then $u = \sum_{j=1}^n t_j u_j$ for some $\sum_{i=1}^n t_j = 1$ and $t_j \geq 0$ for all $j$. For each $u_j$, let $v_j \in \mathcal{V}$ be such that $d(u_j,v_j) \le \delta$, and set $v = \sum_{j=1}^n t_j v_j = u + \sum_{j=1}^n t_j(v_j - u_j)$. Then $|u-v|= \big|\sum_{j=1}^n t_j(v_i - u_i) \big| \leq \delta$ by the triangle inequality. Hence, $\tilde d_H(\mathcal{U},\mathcal{V})\le \delta$. If, in addition $\mathcal{V} \subset \mathcal{U}$, then $\tilde d_H(\mathcal{V},\mathcal{U})=0$, implying that $d_H(\mathcal{U},\mathcal{V})\le \delta$.
\end{proof}

\subsection{Extremal Polynomials}\label{sub:poly} Chebyshev polynomials $T_m(z)$ often play a key role in the analysis of Krylov subspaces for Hermitian matrices. One key property is that $T_m(x)$ achieves the so-called Remez inequality for intervals -- an upper bound on the infinity norm of a polynomial that is bounded by one on some subset of fixed measure (see \cite{remez1936propriete} for details). In our analysis, we  require polynomials with similar extremal properties for segments of the circle and the half disk in the complex plane. For the unit circle, we give the following result.

\begin{proposition}\label{prop:circle_poly}
Let $m \in \mathbb{N}$, $-1\le c_1 < c_2 \le 1$, $0\le \delta<1$, and 
\[\widehat P_{m,c_1,c_2,\delta}(z) = z^m \, T_m\left( \frac{z+z^{-1} -2 c_1}{(1-\delta)(c_2-c_1)} -1\right).\] Then $\widehat P_{m,c_1,c_2,\delta} \in \mathcal{P}_{2m}$ and
\[ \big|\widehat P_{m,c_1,c_2,\delta}(e^{\arccos(c_2) i}) \big| \ge \frac{1}{2} \exp\big(2 m \sqrt{\delta}\big) \max_{ \cos \theta \in [c_1,c_2 - \delta(c_2-c_1)]} \big|\widehat P_{m,c_1,c_2,\delta}(e^{\theta i})\big|. \]
\end{proposition}

\begin{proof}
Clearly, $\widehat P_{m,c_1,c_2,\delta} \in \mathcal{P}_{2m}$. To prove the above inequality, it suffices to show that $|\widehat P_{m,c_1,c_2,\delta}(e^{\theta i})| \le 1$ for $\cos \theta \in [c_1,c_2 - \delta(c_2-c_1)]$, and $|\widehat P_{m,c_1,c_2,\delta}(e^{\arccos(c_2) i})| \ge \frac{1}{2} \exp\big(2 m \sqrt{\delta}\big)$. That $|\widehat P_{m,c_1,c_2\delta}(e^{\theta i})| \le 1$ for $\cos \theta \in [c_1,c_2 - \delta(c_2-c_1)]$ follows immediately from the property $|T_m(x)| \le 1$ for all $x \in [-1,1]$. The lower bound for $e^{\arccos(c_2) i}$ follows from \cite[proof of Lemma 3.9]{urschel2021uniform}.
\end{proof}

In addition, for a fixed $\delta>0$ and degree $m$, we need a polynomial that has modulus at most one in the half annulus
\begin{equation}\label{eqn:halfannulus}
D_\delta = \{ z \in \mathbb{C} \, | \, \delta \le |z|\le 1, \mathrm{Re}(z) \le 0 \} 
\end{equation}
and is as large as possible at $z = 0$. A nearly extremal (up to exponential constants) polynomial was constructed by Erdelyi, Li, and Saff for the unit disk. 

\begin{proposition}[{\cite[proof of Theorem 2.6]{erdelyi1994remez}}]\label{prop:remez_disk}
Let $m \in \mathbb{N}$, $0<\epsilon<1$, $Q_{m,\epsilon}(z) = z^{2m} \, T_m \left( \frac{z + z^{-1}}{2 \cos \epsilon} \right)$, and 
        \[R_\epsilon = \{z \in \mathbb{C} \,\big|\, |z| \leq 1, \arg{z} \in [\epsilon, \pi - \epsilon] \cup [\pi + \epsilon, 2 \pi - \epsilon]\} \cup \{z \in \mathbb{C} \,\big|\, |z| \leq 1-\epsilon/8 \}.\] Then $Q_{m,\epsilon} \in \mathcal{P}_{3m}$ and
        \[Q_{m, \epsilon}(1) \geq \exp\left(\frac{m \epsilon}{8}\right) \max_{z \in R_\epsilon} |Q_{m, \epsilon}(z)|.\]
\end{proposition}

This construction can be modified to $D_\delta$ by applying a simple quadratic transformation.

\begin{proposition}\label{prop:remez}
Let $m \in \mathbb{N}$, $0 \le \delta <1$, $P(z) = (1-\tfrac{\delta}{4}) z^2 + (1-\tfrac{\delta}{8})z + 1$, and
\[\widetilde P_{m,\delta}(z) = P(z)^{2m} \, T_m \left( \frac{P(z) + P(z)^{-1} }{2 \cos \left(\tfrac{2}{3} \delta  \right)}  \right).\]
Then $\widetilde P_{m,\delta} \in \mathcal{P}_{6m}$ and
\begin{equation}\label{ineq:remez}
\widetilde P_{m,\delta}(1) \ge \exp\left(\frac{m \delta}{12}  \right) \max_{z \in D_\delta} |\widetilde P_{m,\delta}(z)|.
\end{equation}
\end{proposition}

Given Proposition \ref{prop:remez_disk}, to prove Proposition \ref{prop:remez}, it suffices to set $\epsilon = \tfrac{2}{3} \delta$ and show that $P(0) = 1$ and $P(D_\delta) \subset R_\epsilon$. This requires some fairly lengthy and un-insightful analysis and casework, and so we defer the proof of Proposition \ref{prop:remez} to the Appendix. See Figure \ref{fig:remez} for an illustration.

In addition, we note that the construction in Proposition \ref{prop:remez} is useful not only for the gap-independent analysis presented in this work, but also for existing gap-dependent estimates as well. For instance, the following is an immediate consequence of Proposition \ref{prop:remez} and \cite[Lemma 6.2]{saad2011numerical}.

\begin{proposition}
Let $A \in \mathbb{C}^{n \times n}$ be a diagonalizable matrix with eigendecomposition $A = V \Lambda V^{-1}$, where the columns of $V$ have norm one. If $\Lambda_j \not \in \mathrm{conv}(\Lambda(A)\backslash \Lambda_j)$  and $d(\Lambda_j,\Lambda(A) \backslash \Lambda_j) \ge \epsilon \, \mathrm{diam}(\Lambda(A))$, then
\[\| (I - P_m) V \bm{e}_j\|_2 \le \frac{\|V^{-1} \bm{b} \|_2^2}{|[V^{-1} \bm{b}]_j|^2} \exp\left(-\frac{(m-1) \epsilon}{72}\right), \]
where $P_m$ is the orthogonal projection onto $\mathcal{K}_{m}(A,\bm{b})$.
\end{proposition}

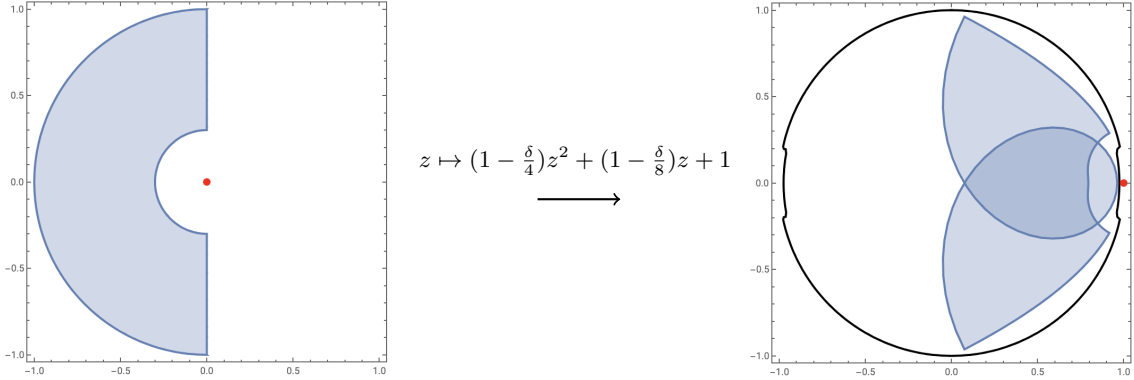
\begin{figure}
\input{figures/regions.tikz}
\caption{The mapping from eigenvalues to inputs for Remez-type polynomials. The desired extremal eigenvalue is shown in red, and the blue region represents all possible eigenvalues at least $\delta$ distance away. The illustrated transformation maps the blue region $D_\delta$ into the outlined region $R_\varepsilon$ from Proposition \ref{prop:remez_disk}
\label{fig:remez}}
\end{figure}

\subsection{Probabilistic Inequalities}\label{sub:probab} To properly analyze the behavior of a Krylov subspace for a random $\bm b \sim \mathrm{Unif}(\mathbb{S}^{n-1})$ on the complex hypersphere, or, equivalently, a complex Gaussian vector $\bm b \sim \compnormvec$, we recall the following fairly standard results.

\begin{proposition}[{\cite[Example 2.11]{wainwright2019high}}]\label{prop:chisq}
Let $X \sim \chi^2_k$ be a chi-squared random variable with $k$ degrees of freedom. Then $\Prob[ |X-k| \ge k t] \le 2 e^{-kt^2/8}$ for all $t \in (0,1)$.
\end{proposition}

\begin{proposition}\label{prop:perturbed_basis}
Let $M \in \mathbb{C}^{n \times n}$ be invertible, $\bm b \sim \mathrm{Unif}(\mathbb{S}^{n-1})$, and $t \in (0,1)$. Then
\[ \Prob \left[ \min_{j \in [n]} \frac{|[ M \bm b]_j|^2}{\|M \bm b \|^2_2} \ge \frac{t}{n^2 \kappa^2(M)}\right] \ge 1 - et.\]
\end{proposition}

\begin{proof}
By \cite[Lemma 2.2]{dasgupta2003elementary}, $\Prob\big[ |\langle \bm b , \bm \nu \rangle | ^2 \le \frac{\tau}{n} \big] \le e \tau$ for $\tau \in [0,1)$ and $\|\bm \nu \|_2 = 1$. We have
\[\Prob \left[|\langle M \bm b, \bm \nu \rangle|^2 \le \frac{t}{n^2 \|M^{-1}\|_2^2}\right] \le \Prob \left[|\langle \bm b, M^* \bm \nu \rangle|^2 \le \frac{t \| M^* \bm \nu\|_2^2}{n^2 }\right] \le \frac{et}{n} .\]
Applying a union bound for $\bm{e}_j$, $j \in [n]$, and noting that $\|M \bm{b}\|_2^2 \le \|M\|_2^2$ implies our desired result.
\end{proof}

In addition, the following propositions will be used to treat the non-normal setting in Section \ref{sec:nonnormal}.
\begin{proposition}\label{prop:traces}
    Let $\bm{b} \sim \compnormvec$ and $M \in \mathbb{C}^{n\times n}$ be a normal matrix.
    Then, $\mathbb{E}[\bm{b}^*M \bm{b}] = \mathrm{trace}(M)$ and $\mathrm{Var}[\bm{b}^*M \bm{b}] = \|M\|_F^2.$ 
\end{proposition}
\begin{proof}
We provide a proof only because of the involvement of complex random variables; see \cite[Lemma 3.2b.2]{quadraticforms} for well-known results for the real symmetric case. When $M$ is normal, the invariance of the distribution of $\bm{b}$ under rotation allows us to rewrite $\bm{b}^*M \bm{b} \overset{d}{=} \sum_{j=1}^n \lambda_j |\bm{b}_j|^2$, where $\{\lambda_j\}_{j=1}^n$ are the eigenvalues of $M$. For $\bm{b}_j \sim \compnorm$, we have $\mathbb{E}[|\bm{b}_j|^2] = 1$ and $\mathrm{Var}[|\bm{b}_j|^2] = 1$, and so
    \[\mathbb{E}\bigg[\sum_{j=1}^n \lambda_j |\bm{b}_j|^2 \bigg] = \sum_{j=1}^n \lambda_j \mathbb{E}\big[|\bm{b}_j|^2 \big] = \sum_{j=1}^n \lambda_j\]
    and
    \[ \mathrm{Var}\bigg[ \sum_{j=1}^n \lambda_j |\bm{b}_j|^2 \bigg] = \sum_{j=1}^n |\lambda_j|^2 \mathrm{Var}[|\bm{b}_j|^2] = \sum_{j=1}^n |\lambda_j|^2. \]
\end{proof}

\begin{proposition}[{\cite[Example 2.8, Proposition 2.9]{wainwright2019high}}]\label{cor:subexp}
Let $\bm{b} \sim \compnormvec$ and $M \in \mathbb{C}^{n\times n}$ be a Hermitian or skew-Hermitian matrix. Let $X = \bm{b}^*M \bm{b}$. Then,
$$\prb{|X - \mathbb{E}[X]| \geq t} \leq 2 \exp \left[ -\min \left( \frac{t^2}{4\|M\|_F^2}, \frac{t}{4\|M\|_2} \right)\right].$$
\end{proposition}
\begin{proof}
    Again rewrite $\bm{b}^*M \bm{b} \overset{d}{=} \sum_{j=1}^n \lambda_j |\bm{b}_j|^2$ where $\lambda_j$ are the eigenvalues of $M$. Since $M$ is Hermitian or skew-Hermitian, the eigenvalues are purely real or imaginary. Each term $|\lambda_j||\bm{b}_j|^2$ is sub-exponential with parameters $(\sqrt{2}|\lambda_j|, 2|\lambda_j|)$; thus, $\sum_{j=1}^n |\lambda_j||\bm{b}_j|^2$ is sub-exponential with parameters $(\sqrt{2}\|M\|_F, 2\|M\|_2)$.
\end{proof}

\begin{corollary}\label{cor:nonnormalsubexp}
    Let $M \in \mathbb{C}^{n \times n}$, $M_1 = \frac{1}{2}(M + M^*)$, and $M_2 = \frac{1}{2}(M-M^*)$, with corresponding random variables $X = \bm b^* M \bm b$, $X_1 = \bm b^*M_1\bm b$, and $X_2 =\bm b^*M_2 \bm b$. Then \[\mathrm{Var}[X] = \|M_1\|_F^2 + \|M_2\|_F^2 \qquad \text{ and } \qquad \prb{|X - \mathbb{E}[X]| \geq t} \leq 4 \exp \left[ -\min \left( \frac{t^2}{8\mathrm{Var}[X]}, \frac{t}{4 \sqrt{2} \|M\|_2} \right)\right].\]
\end{corollary}
\begin{proof}
    Note that $X = X_1 + X_2$, $X_1$ is purely real, and $X_2$ is purely imaginary. For the first part, rewrite 
    \[\mathrm{Var}[X] = \mathrm{Var}[\mathrm{Re}(X)] + \mathrm{Var}[\mathrm{Im}(X)] = \mathrm{Var}[X_1] + \mathrm{Var}[X_2].\] For the second, note that
    \[\prb{|X - \mathbb{E}[X]| \geq t} \leq \prb{|X_1 - \mathbb{E}[X_1]| \geq t/\sqrt{2}} + \prb{|X_2 - \mathbb{E}[X_2]| \geq t/\sqrt{2}}.\] The given bound is not tight but suffices for our purposes.
\end{proof}

\begin{proposition}\label{prop:notsmall}
    Let $\bm b \sim \compnormvec$ be a random vector and $M \in \mathbb{C}^{n\times n}$ be a positive-definite Hermitian matrix. Then,
    $\prb{\bm b^*M \bm b \leq t \cdot \mathrm{trace}(M)} \leq et$ for all $t > 0$.
\end{proposition}
\begin{proof}
    The proof is a standard Chernoff bound adapted from \cite{anticoncentration} for complex normals. Again rewrite $\bm{b}^*M \bm{b} \overset{d}{=} \sum_{j=1}^n \lambda_j |\bm{b}_j|^2$ where $\lambda_j$ are the eigenvalues of $M$. Assume $t<1$, otherwise the statement is trivially true. Then, for $c = (t^{-1}-1)/\mathrm{trace}(M)$,
    \begin{align*}
        \prb{\bm b^*M\bm b \leq t \cdot \mathrm{trace}(M)} &= \prb{ t \sum_{j=1}^n \lambda_j - \sum_{j=1}^n \lambda_j |\bm{b}_j|^2 \geq 0}  \\ 
        &= \prb{\prod_{j=1}^n e^{c \lambda_j(t - |\bm{b}_j|^2)} \geq 1}\\
        &\leq \prod_{j=1}^n \mathbb{E}\left[e^{c \lambda_j(t - |\bm{b}_j|^2)} \right] \\
        &= \prod_{j=1}^n \frac{e^{c \lambda_j t}}{1+c \lambda_j} \\
        &\leq e^{c\cdot \mathrm{trace}(M) t} (1+c \cdot \mathrm{trace}(M))^{-1} = e^{1-t} t.
    \end{align*}
\end{proof}

\section{The Numerical Range of a Normal Matrix} \label{sec:normal}
In this section, we produce estimates on the approximation of the numerical range $W(A)$ of a normal matrix $A \in \mathbb{C}^{n \times n}$ by a Krylov subspace. In particular, we prove estimates on the Hausdorff distance between $W(A)$ and $W(H_m)$, where $H_m$ is the orthogonal projection of $A$ onto $\mathcal{K}_m(A,\bm b)$ for a random $\bm b \sim \mathrm{Unif}(\mathbb{S}^{n-1})$. We produce an upper bound of the form $\tfrac{\ln n}{m}$ for an arbitrary normal matrix (Subsection \ref{sub:num_upper}), and a nearly matching lower bound of the form $\tfrac{1}{m}$ (Subsection \ref{sub:num_lower}). (The necessity of the log term has already been treated in a number of previous works, see \cite[Theorem A.1]{simchowitz2018tight} and \cite[Lemma 3.8]{urschel2021uniform}.) We also provide an improved estimate of the form $\tfrac{\ln^2 n}{m^2}$ for normal matrices with their spectrum on the unit circle $\mathbb{S}^0$, and a nearly matching lower bound of the form $\tfrac{1}{m^2}$ (Subsection \ref{sub:num_circle}). We note that, based on the degree of the polynomials constructed in Subsection \ref{sub:poly}, the below results are stated not for all $m$, but for either $2m+1$ or $6m+1$. However, this immediately implies bounds for arbitrary $m$, based on the standard inclusion $\mathcal{K}_1(A,\bm b) \subset \mathcal{K}_2(A,\bm b) \subset \ldots \subset \mathcal{K}_m(A,\bm b) \subset \ldots$

\subsection{Estimating the Numerical Range}\label{sub:num_upper}

To estimate the approximation $W(H_m)$ provides to $W(A)$, we first estimate the distance between $W(H_m)$ and an arbitrary extreme eigenvalue $\lambda$ of $A$.

\begin{lemma}\label{lm:single_eig}
Let $m,n \in \mathbb{N}$, $A \in \mathbb{C}^{n \times n}$ be normal, and $\lambda \in \Lambda(A)$ be an extreme point of $W(A)$ with a corresponding unit eigenvector $\bm{\varphi}$. Then
\[  \min_{\bm x \in \mathcal{K}_{6m+1}(A,\bm b)} \bigg| \lambda - \frac{\bm x^* A \bm x}{\bm x^* \bm x} \bigg| \le \frac{6}{m} \ln  \bigg( \frac{em \|\bm b\|^2_2 }{6 |\langle \bm b, \bm \varphi \rangle |^2}\bigg) \, \mathrm{diam}(W(A)). \]
\end{lemma}

\begin{proof}
Without loss of generality, let $\lambda = 0$ be the extreme point of $W(A)$ of interest, $\mathrm{diam}(W(A)) =1$, and $W(A) \subset D_0$, the half disk (see Equation \ref{eqn:halfannulus}). Let $\delta >0$ and $S_\delta = \{\Lambda_j \in \Lambda(A) \, | \, |\Lambda_j| > \delta \}$. By Equation \ref{eqn:normal_RQ},
\begin{align*}
\min_{x \in \mathcal{K}_{6m+1}(A,b)} \bigg| \frac{\bm x^* A \bm x}{\bm x^* \bm x} \bigg|  &\le  \min_{p \in \mathcal{P}_{6m}} \left| \frac{\sum_{\Lambda_j \in S_\delta} |\bm \alpha_j|^2  |p(\Lambda_j)|^2 \Lambda_j}{\sum_{j=1}^n |\bm \alpha_j|^2  |p(\Lambda_j)|^2} \right| +  \left| \frac{\sum_{\Lambda_j \not \in S_\delta} |\bm \alpha_j|^2  |p(\Lambda_j)|^2 \Lambda_j}{\sum_{j=1}^n |\bm \alpha_j|^2  |p(\Lambda_j)|^2} \right| \\
&\le \min_{p \in \mathcal{P}_{6m}} \frac{ \sum_{\Lambda_j \in S_\delta} |\bm \alpha_j|^2  |p(\Lambda_j)|^2 }{|\langle \bm b, \bm \varphi \rangle |^2  |p(0)|^2} + \delta.
\end{align*}
Now, consider the choice $p(z) = \widetilde P_{m, \delta}(z)$, as defined in Proposition \ref{prop:remez}. By Inequality \ref{ineq:remez},
\[\frac{\sum_{\Lambda_j \in S_\delta} |\bm \alpha_j|^2  |\widetilde P_{m,\delta}(\Lambda_j)|^2 }{|\langle \bm b, \bm \varphi \rangle |^2  |\widetilde P_{m,\delta}(0)|^2} \le \frac{\|\bm b\|^2_2}{|\langle \bm b, \bm \varphi \rangle |^2} \exp\left(-\frac{m \delta}{6} \right) .\]
 Setting $\delta = 6 m^{-1} \ln  \big( \frac{m \| \bm b\|^2_2 }{6 |\langle \bm b, \bm \varphi \rangle |^2}\big)$ (if $\delta >1$, the result is trivially true), we obtain our desired result
\[ \frac{\|\bm b\|^2_2}{|\langle \bm b, \bm \varphi \rangle |^2} \exp\left(-\frac{m \delta}{6} \right) + \delta \le \frac{6}{m}\left(1 +   \ln  \bigg( \frac{m \|\bm b\|^2_2 }{6 |\langle \bm b, \bm \varphi \rangle |^2}\bigg) \right) .\]
\end{proof}

\begin{theorem}\label{thm:main} Let $m,n \in \mathbb{N}$, $\alpha>0$, $A \in \mathbb{C}^{n \times n}$ be normal, and $\{\bm \varphi^{(1)},\ldots,\bm \varphi^{(n)} \}$ be an orthonormal eigenbasis. Then
\[d_{H}\big(W(H_{6m+1}),W(A) \big) \le \frac{6}{m} \ln  \bigg( \frac{em \|\bm b\|^2_2 }{6 \min_{j \in [n]} |\langle \bm b, \bm \varphi^{(j)} \rangle |^2}\bigg) \, \mathrm{diam}(W(A)) . \]
In particular, if $\bm b \sim \mathrm{Unif}(\mathbb{S}^{n-1})$, then
\[\Prob\bigg[ \frac{d_{H}\big(W(H_{6m+1}),W(A) \big)}{\mathrm{diam}(W(A)) } \le \frac{6(2+\alpha) \ln n}{m}  \bigg] \ge 1 - \frac{5m}{4n^{\alpha}}. \]
\end{theorem}

\begin{proof}
By Proposition \ref{prop:polytopevertices} and the inclusion $W(H_{6m+1})\subset W(A)$, if the distance from every extreme point of $W(A)$ to $W(H_{6m+1})$ is at most $\epsilon \, \mathrm{diam}(W(A))$, then the Hausdorff distance between $W(H_{6m+1})$ and $W(A)$ is also at most that quantity. Combining this fact with Lemma \ref{lm:single_eig} completes the proof of the first part of the theorem statement.

Now suppose that $\bm b \sim \mathrm{Unif}(\mathbb{S}^{n-1})$. By Proposition \ref{prop:perturbed_basis} with $M = [\bm \varphi^{(1)} \ldots \, \bm \varphi^{(n)}]^*$ and $t = e m /(6 n^{\alpha})$, $\min_{j \in [n]} | \langle  \bm b, \bm \varphi^{(j)} \rangle|^2 \ge e m /(6 n^{2+\alpha})$ with probability at least $1 - e^2 m /(6 n^\alpha) > 1- \tfrac{5}{4} m n^{-\alpha}$. Therefore, with probability at least $1- \tfrac{5}{4} m n^{-\alpha}$,
\[ \frac{6}{m}  \ln  \bigg( \frac{ e  m \|\bm b\|^2_2 }{6 |\langle \bm b, \bm \varphi^{(j)} \rangle  |^2}\bigg) \le \frac{6}{m}  \ln  n^{2+\alpha} = \frac{6(2+\alpha) \ln n}{m} \qquad \text{for all} \quad j \in [n].  \]
\end{proof}

\subsection{Lower Bounds for Krylov Subspace Approximation}\label{sub:num_lower}
In the previous subsection we proved that, with high probability, the numerical range is approximated up to error $\epsilon$ when $m$ $\scriptstyle \gtrsim$ $\epsilon^{-1} \ln n$. It is already well-known that the Krylov subspace dimension $m$ must be at least $\ln n$ in order to guarantee any non-trivial approximation of the entire numerical range \cite[Theorem A.1]{simchowitz2018tight}. Furthermore, some multiplicative relationship between $\ln n$ and $\epsilon^{-1}$ is needed in the bound for $m$. For instance, taking $\epsilon = [(1+\gamma)\ln \ln n]^{-1}$, $m = \gamma \ln n + P(1/\epsilon)$ for any fixed constant $\gamma>0$ and polynomial $P$, and taking $n$ sufficiently large, the approximation produced will be greater than $\epsilon$ with probability $1-o_n(1)$ \cite[Lemma 3.8]{urschel2021uniform}. In addition, when the matrix is Hermitian, dimension $m \gtrsim \epsilon^{-1/2}$ is required to guarantee $\epsilon$ error, a tight lower bound in this setting \cite[Lemma 3.5]{urschel2021uniform}. Here, we show that for normal matrices $m \gtrsim \epsilon^{-1}$ is required to obtain at most $\epsilon$ error, thus matching the bounds for Theorem \ref{thm:main}. In fact, we show this not just for Hausdorff distance, but for the distance between $W(H_m)$ and $\partial W(A)$, i.e., we produce an instance where no extreme eigenvalue of $A$ is approximated better than $1/m$.

\begin{proposition}\label{prop:power_sum}
Let $k,n \in \mathbb{N}$, $k < n$. Then $\sum_{j=1}^n j^{k+1} < (n-\tfrac{1}{e+1}) \sum_{j=1}^n j^{k}$.
\end{proposition}

\begin{proof}
We have
\[\frac{ n \sum_{j=1}^n j^{k} -  \sum_{j=1}^n j^{k+1}}{\sum_{j=1}^n j^{k}} = \frac{ \sum_{j=1}^n (n-j) j^{k}}{\sum_{j=1}^n j^{k}} \ge \frac{(n-1)^k}{n^k + (n-1)^k} = \frac{1}{\big(1+\tfrac{1}{n-1} \big)^k+1}>\frac{1}{e+1}.\]
\end{proof}

\begin{proposition}\label{prop:polar}
Let $V \in \mathbb{C}^{m \times n}$ satisfy $V^* V = I$ with $m \ge n$, $D \in \mathbb{C}^{m \times m}$ be diagonal with $\|D\|_2  <1$, and $QH$ be the polar decomposition of $(I+D)V$. Then
\[ \| V^* A V - Q^* A Q \|_2 \le \frac{4 \|A\|_2  \|D\|_2}{1-\|D\|_2} \]
for all $A \in \mathbb{C}^{m \times m}$.
\end{proposition}

\begin{proof}
Because $\|D\|_2<1$, the matrix $(I+D)V$ has full column rank, and so 
\[Q = (I+D)V [V^*(I+D^*) (I+D) V]^{-1/2} \in \mathbb{C}^{m \times n}\] in the polar decomposition of $(I+D)V$ is unique \cite[Theorem 7.3.1]{horn2012matrix}. Letting $E = V- Q$,
\begin{align*} \| V^* A V - Q^* A Q \|_2 &= \|E^*AV + Q^*AE \|_2 \\ &\le  \|E\|_2 \|A \|_2 \|V\|_2 + \|Q\|_2 \|A\|_2 \|E\|_2 \\ &= 2 \|E\|_2 \|A \|_2 .
\end{align*}
What remains is to bound $\|E\|_2$. Using a similar technique,
\begin{align*}
\|E\|_2&= \|-DV + (I+D)V\big(I - [V^* (I+D^*) (I+D) V]^{-1/2}\big)\|_2 \\
&\le \|D\|_2 + (1 + \|D\|_2) \|I - [V^* (I+D^*) (I+D) V]^{-1/2}\|_2.
\end{align*}
The matrix $[V^* (I+D^*) (I+D) V]$ has singular values in the interval $[(1-\|D\|_2)^2,(1+\|D\|_2)^2]$, and so 
\[ \|I - [V^* (I+D^*) (I+D) V]^{-1/2}\|_2 \le \frac{1}{1-\|D\|_2}-1 = \frac{\|D\|_2}{1-\|D\|_2}.\]
Combining all these bounds, we obtain
\[\| V^* A V - Q^* A Q \|_2 \le 2 \|A\|_2\bigg( \|D\|_2 + \frac{(1+ \|D\|_2) \|D\|_2}{1-\|D\|_2}\bigg) 
= \frac{4 \|A\|_2  \|D\|_2}{1-\|D\|_2}.\]
\end{proof}

\begin{theorem}\label{thm:norm_lower}
Let $n = \ell \, m^2$, $\ell, m \in \mathbb{N}$, $m \ge 15$, $\ell \ge 4800 \, m^2  \ln m $, $\bm 1 = (1,\ldots,1)^T \in \mathbb{C}^{\ell}$, $\bm r = \tfrac{1}{\sqrt{m}} (\sqrt{1},\sqrt{2},\ldots,\sqrt{m})^T \in \mathbb{C}^m$, $\omega = \exp\big(\frac{2 \pi i}{m}\big)$, $\bm \theta = (\omega,\omega^2,\ldots,\omega^m)^T \in \mathbb{C}^{m}$, and $A = \mathrm{diag}( \bm r \otimes \bm \theta \otimes \bm 1) \in \mathbb{C}^{n\times n}$. If  $\bm b \sim \mathrm{Unif}(\mathbb{S}^{n-1})$, then
\[ \Prob \bigg[ \frac{d\big(W(H_{m}), \partial W(A) \big)}{\mathrm{diam}(W(A)) } \ge \frac{1}{60 m}  \bigg] \ge 1 - \frac{2}{m} . \]
\end{theorem}

\begin{proof}
Without loss of generality, we may take $\bm b \sim \compnormvec$ instead of $\bm b \sim \mathrm{Unif}(\mathbb{S}^{n-1})$. Let \[\hat A = \mathrm{diag}(\bm r \otimes \bm \theta) \in \mathbb{C}^{m^2 \times m^2},\] $ \hat {\bm b} \in \mathbb{R}^{m^2}$ be such that 
\[\hat {\bm b}_j= \sqrt{\sum_{k = 1}^{\ell} |\bm b_{j+(k-1)m^2}|^2},\] and $\hat H_m = \hat Q_m^* \hat A \hat Q_m \in \mathbb{C}^{m \times m}$ denote the orthogonal projection of $\hat A$ onto $\mathcal{K}_{m}(\hat A,\hat {\bm b})$ in the basis of $\hat Q_m$. We note that, because $A$ and $\hat A$ have $m^2$ distinct eigenvalues and $\bm b \sim \compnormvec$, $\mathcal{K}_{m}(A,\bm b)$ and $\mathcal{K}_{m}(\hat A,\hat{\bm b})$ both have dimension $m$ with probability one, and so $\hat Q_m$ is well-defined. By Equation \ref{eqn:normal_RQ}, $H_m$, the orthogonal projection of $A$ onto $\mathcal{K}_{m}(A,b)$, and $\hat H_m$ have the same Rayleigh quotient for each polynomial $p \in \mathcal{P}_{m-1}$, and so have the same numerical range. Therefore, we may work with $\hat H_m$ instead of $H_m$. Note that two times the squared entries of $\hat {\bm b}$ are independent chi-squared variables with $2\ell$ degrees of freedom each. 

Now, let $\tilde {\bm b} = \sqrt{\ell}\,  \bm{1}\in \mathbb{C}^{m^2}$ and $\tilde H_m = \tilde Q_m^* \hat A \tilde Q_m$ be the orthogonal projection of $\hat A$ onto $\mathcal{K}_{m}(\hat A,\tilde {\bm b})$ (i.e., $\hat H_m$ is a ``noisy" version of $\tilde H_m$). The benefit of $\tilde H_m$ is that the vectors $\hat A^j \tilde {\bm b}$ are orthogonal and have easily computable norms:
  \begin{align*}    \langle \hat A^j \tilde {\bm b}, \hat A^k \tilde {\bm b}\rangle &= \ell \sum_{p=1}^{m} \sum_{q=1}^{m}  \bm{\theta}_p^{j}\bm{r}_q^j \bm{\theta}_p^{-k}\bm{r}_q^k  \\ &=  \ell \sum_{p=1}^m e^{2 \pi i(j-k)\frac{p}{m}} \sum_{q=1}^m\left(\frac{q}{m}\right)^{\frac{j+k}{2}} \\
  &= \begin{cases}
            \ell m \sum_{q=1}^{m} \left(\frac{q}{m} \right)^{j} &  \text{ if } j=k \\
            0 & \text{ otherwise }.
        \end{cases}
        \end{align*}
Therefore, $[\tilde H_m]_{jk} = 0$ for $j \ne k+1$, and, by Proposition \ref{prop:power_sum},
\[\tilde H_{k+1,k} = \frac{\|\hat A^{k} \tilde {\bm b}\|}{\|\hat A^{k-1} \tilde {\bm b}\|} = \left(\frac{ \sum_{q=1}^{m} q^{k} }{m \sum_{q=1}^{m} q^{k-1} }\right)^{1/2} \le \left(1- \frac{1}{(e+1)m}\right)^{1/2} .\]
By Gershgorin's Disk Theorem, the matrix $(e^{i \phi} \tilde H_m + e^{- i \phi} \tilde H_m^*)/2$ has eigenvalues all bounded by $\big(1- \tfrac{1}{(e+1)m}\big)^{1/2}$ in modulus for all $\phi \in [0,2\pi)$, implying that the maximum modulus of an element of $W(\tilde H_m)$ is at most $\big(1- \tfrac{1}{(e+1)m}\big)^{1/2}$. Now, let us consider how different the situation can be for $\hat H_m$.  For $\hat {\bm b}$, we note that the vectors $\hat A^j \hat {\bm b}$, $j = 0,\ldots,m-1$, have the same span as the columns of $(I+D) \tilde Q_m$, where $D$ is diagonal with $D_{jj} = \frac{\hat  {\bm{b}}_j}{\sqrt{\ell}} - 1$.  By Proposition \ref{prop:chisq} with $t:= \sqrt{12 \ell^{-1} \ln m}<1$ and a union bound,
\[\mathbb{P}\left[ \max_{j =1,\ldots,m^2} |2\hat{\bm b}_j^2 - 2\ell| \ge \sqrt{48 \ell \ln m}  \right] \le \frac{2}{m}.\]
Therefore, for $\ell \ge 4800 m^2 \ln m$,
\[ \|D\|_2 \le 1-  \sqrt{1 - \sqrt{\frac{12 \ln m}{\ell} }} \le 1-  \sqrt{1 - \frac{1}{20m}}   \]
with probability at least $1 -2/m$. By Proposition \ref{prop:polar}, applied to $\tilde Q_m$ and $(I+D)\tilde Q_m$ (with $V:= \tilde Q_m$),
\[\|\tilde H_m - Q^* \hat A Q\|_2 = \| \tilde Q_m^* \hat A \tilde Q_m - Q^* \hat A Q\|_2  \le \frac{4\left(1-  \sqrt{1 - \frac{1}{20m}}  \right)}{\sqrt{1 - \frac{1}{20m}}},\]
where $Q \in \mathbb{C}^{m^2 \times m}$ is as in Proposition \ref{prop:polar}, i.e., $Q$ is the matrix with orthonormal columns in the polar decomposition of $(I+D)\tilde Q_m$. This means that $Q$ and $\hat Q_m$ have the same column space, and so $Q = \hat Q_m U$ for some unitary matrix $U \in \mathbb{C}^{m \times m}$, and 
\[\|\tilde H_m - Q^* \hat A Q\|_2 = \|\tilde H_m - U^* \hat Q_m^* \hat A \hat Q_m U\|_2 = \| \tilde H_m - U^* \hat H_m U \|_2.\]
Therefore, the largest modulus of an element of $W(\hat H_m)$ is at most 
\[f(m):=\left(1- \frac{1}{(e+1)m}\right)^{1/2} + \frac{4\left(1-  \sqrt{1 - \frac{1}{20m}}  \right)}{\sqrt{1 - \frac{1}{20m}}} .\]
The function $(1-f(m))m$ is monotonically decreasing, with limit $(4-e)/(10(1+e))>1/30$ as $m \rightarrow \infty$. Therefore, $f(m)< 1 - 1/(30m)$, completing the proof.
\end{proof}

\begin{remark}\label{rem:lower}
The constants used in the lower bound of $1/m$ in Theorem \ref{thm:norm_lower} are of moderate size. A similar theorem, with a lower bound of $1/m$ just for Hausdorff distance (instead of distance to boundary $\partial W(A)$), can be produced with much milder requirements on the size of $n$: essentially, just take the same eigenvalues as in Theorem \ref{thm:norm_lower}, give $\lambda = 1$ very high multiplicity, and apply the same technique used in \cite[Lemma 3.5]{urschel2021uniform}. However, we note that the requirements in Theorem \ref{thm:norm_lower} are purely for the sake of producing a clean and concise mathematical proof; practice shows that this phenomenon actually occurs for relatively small $n$. See Figure \ref{fig:lower}.
\end{remark}

\begin{figure}
\subfigure[Zoomed in numerical range with marked eigenvalues]{\includegraphics[scale=0.45]{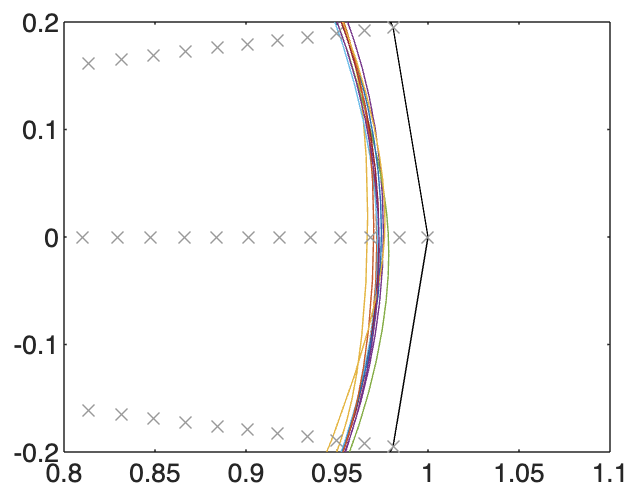}}
\subfigure[Error as $m$ increases]{\includegraphics[scale=0.45]{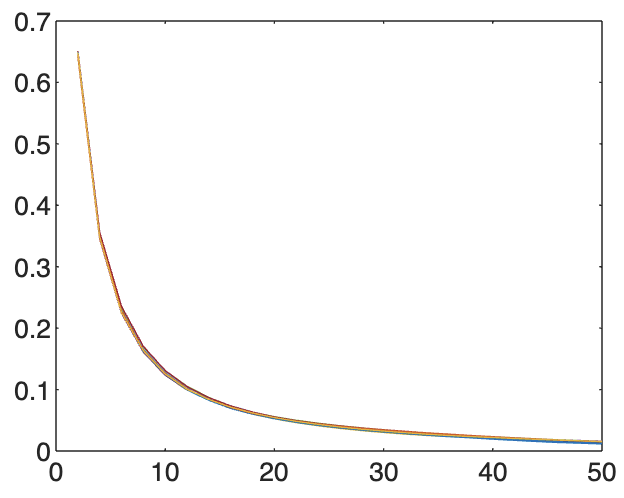}}
\subfigure[Error normalized times $m$]{\includegraphics[scale=0.45]{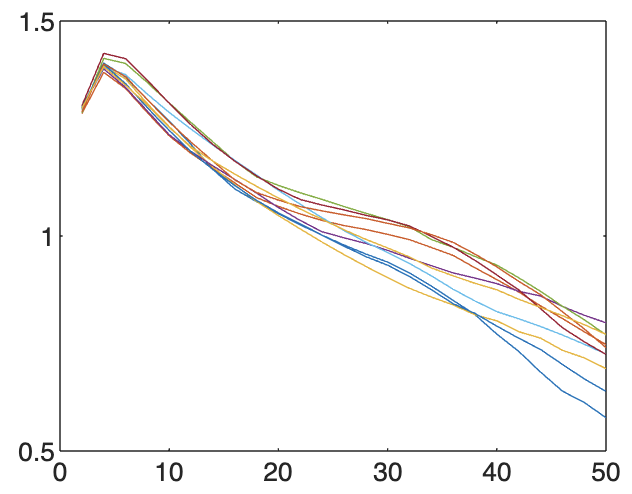}} \\
\subfigure[Zoomed in numerical range]{\includegraphics[scale=0.46]{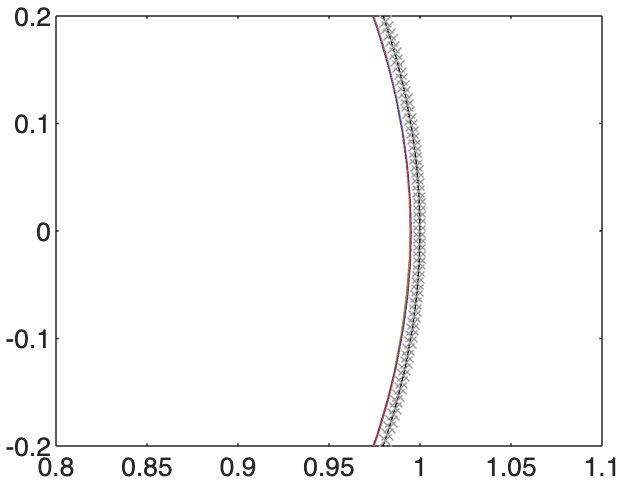}}
\subfigure[Error as $m$ increases]{\includegraphics[scale=0.38]{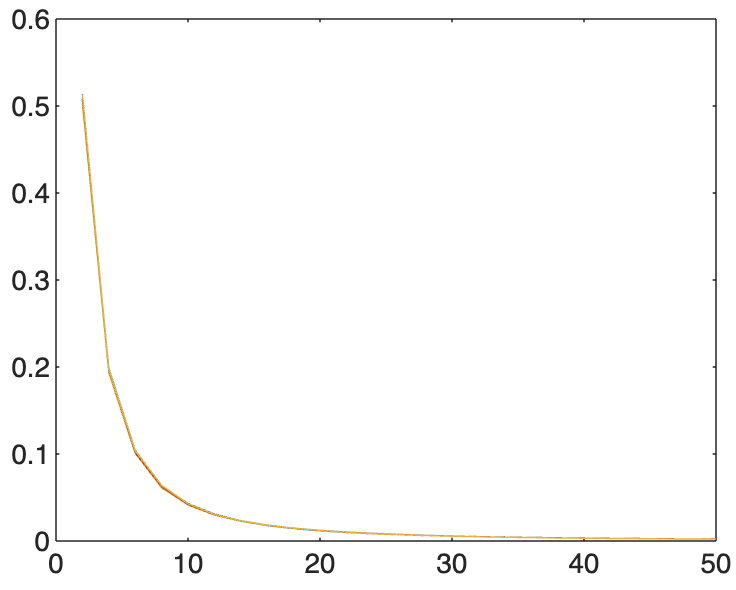}}
\subfigure[Error normalized times $m^2$]{\includegraphics[scale=0.38]{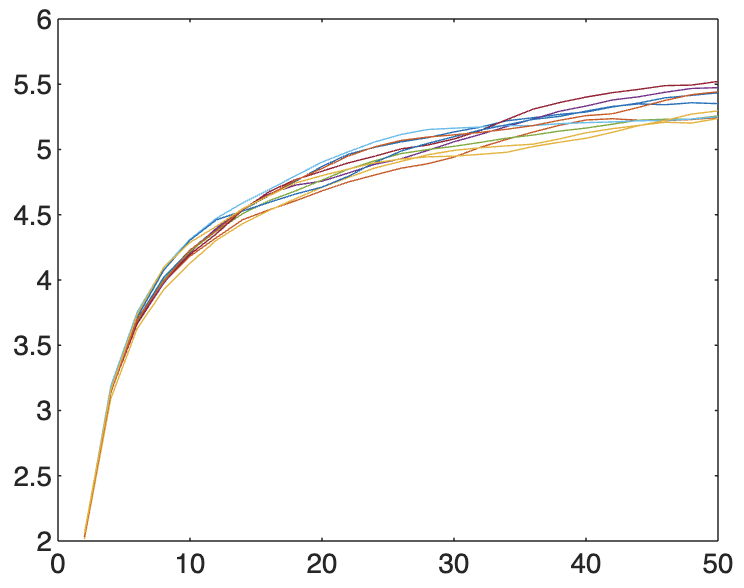}}

\caption{Numerical range estimates of $A$ from Theorem \ref{thm:norm_lower} (top row) and Theorem \ref{thm:circle_upper} (bottom row). Whereas $n$ changes as a function of $m$ to simplify the proof, the essence of the statement holds for each matrix individually. Here we fix $A \in \mathbb{C}^{n \times n}$ with $n=10240$ with eigenvalue multiplicity $10$ and run Arnoldi for $m$ up to 50. In Subfigures (a) and (d), we plot part of the numerical range of $A$ along with the estimated numerical range for 10 random starting vectors $\bm{b}$ and $m=50$. In Subfigures (b) and (e), we plot the Hausdorff distance between the estimated and actual numerical ranges as $m$ increases. To better see the $1/m$ and $1/m^2$ behavior respectively, we multiply the error by $m$ in Subfigure (c) and by $m^2$ in Subfigure (f).}
\label{fig:lower}
\end{figure}

\subsection{Spectrum on the Circle}\label{sub:num_circle}

Next, we consider the situation where the spectrum of $A$ lies on the complex unit circle $\mathbb{S}^0$. A common example of such matrices is the unitary group. We provide an improved upper bound of the form $\tfrac{\ln^2 n}{m^2}$ and a nearly matching lower bound of the form $\tfrac{1}{m^2}$.

\begin{theorem}\label{thm:circle_upper}
Let $m,n \in \mathbb{N}$, $\alpha>0$, and $A \in \mathbb{C}^{n \times n}$ be normal, with its spectrum on $\mathbb{S}^0$ and an orthonormal eigenbasis $\{ \bm \varphi^{(1)},\ldots,\bm \varphi^{(n)}\}$. Then
\[d_{H}\big(\mathrm{Re}(e^{i \theta} W(H_{2m+1}),\mathrm{Re}(e^{i \theta} W(A)) \big) \le \frac{1}{16 m^2} \ln^2 \left( \frac{64 e m^2 \| \bm b \|^2}{\min_{j \in [n]} |\langle \bm b , \bm \varphi^{(j)} \rangle |^2}  \right) \quad \text{for all} \; \theta \in [0,2\pi).\]
In particular, if $\bm b \sim \mathrm{Unif}(\mathbb{S}^{n-1})$, then
\[\Prob \bigg[d_{H}\big(\mathrm{Re}(e^{i \theta} W(H_{2m+1}),\mathrm{Re}(e^{i \theta} W(A)) \big) \le \frac{(2+\alpha)^2  \ln^2 n }{16 m^2} \quad \text{for all} \; \theta \in [0,2\pi) \bigg] \ge 1 - \frac{64 e^2 m^2}{n^{\alpha}}. \]
\end{theorem}

\begin{proof}
Our general proof technique is as follows. First, we produce a bound on the difference between the orthogonal projections of $W(A)$ and $W(H_{2m+1})$ onto an arbitrary line. Then, we lower bound the magnitude of the orthogonal projection of our vector $\bm b$ onto eigenvectors.

Without loss of generality, consider the projection of $W(H_m)$ onto the real line. Let $\mathrm{Re}(W(A)) = [c_1,c_2]$, $\lambda$ be an eigenvalue of $A$ with largest real part, and $\bm \varphi$ be a corresponding eigenvector. Similar to the proof of Lemma \ref{lm:single_eig}, we have
\begin{align*}
\min_{x \in \mathcal{K}_{2m+1}(A,b)} \mathrm{Re}\left( \frac{\bm x^* (c_2 I-A) \bm x}{(c_2-c_1)\bm x^* \bm x} \right) &= \min_{P \in \mathcal{P}_{2m}} \frac{\sum_{j=1}^n |\bm \alpha_j|^2 |P(\lambda_j)|^2 \mathrm{Re}(c_2-\lambda_j)}{(c_2-c_1)\sum_{j=1}^n |\bm \alpha_j|^2 |P(\lambda_j)|^2} \\
&\le \delta + \min_{P \in \mathcal{P}_{2m}} \frac{\sum_{\mathrm{Re}(c_2-\lambda_j)> \delta(c_2-c_1)} |\bm \alpha_j|^2 |P(\lambda_j)|^2}{|\langle \bm b, \bm \varphi \rangle|^2 |P(\lambda)|^2}.
\end{align*}
Consider the choice $\widehat P_{m,c_1,c_2,\delta}(z)$, defined in Proposition \ref{prop:circle_poly}, with $\delta = \frac{1}{4m^2} \ln^2 \left( \frac{64 m^2 \| \bm b \|^2}{|\langle \bm b, \bm \varphi \rangle |^2} \right)$. Then
\[ \delta +  \min_{P \in \mathcal{P}_{2m}} \frac{\sum_{\mathrm{Re}(c_2-\lambda_i)> \delta(c_2-c_1)} |\alpha_i|^2 |P(\lambda_i)|^2}{|\langle \bm b, \bm \varphi \rangle|^2 |P(\lambda)|^2}\le  \delta + \frac{\|\bm b\|^2}{|\langle \bm b, \bm \varphi \rangle|^2 \frac{1}{4}e^{4m \sqrt{\delta}}} \le \frac{1}{16 m^2} \ln^2 \left( \frac{64 e m^2 \| \bm b \|^2}{|\langle \bm b , \bm \varphi \rangle |^2}  \right) .\]
By Proposition \ref{prop:perturbed_basis} with $M = [\bm \varphi^{(1)} \ldots \, \bm \varphi^{(n)} ]$ and $t = 64 e m^2/n^{\alpha}$, $\min_{j \in [n]} | \langle  \bm b , \bm \varphi^{(j)} \rangle|^2 \ge 64 e m^2 /n^{2+\alpha}$ with probability at least $1 - 64 e^2 m^2 /n^\alpha $. Therefore, with probability at least $1- 64 e^2 m^2 n^{-\alpha}$, the relative difference in projection between $W(A)$ and $W(H_m)$ on any line is at most $\tfrac{(2+\alpha)^2}{16} m^{-2}  \ln^2 n$ at either endpoint.
\end{proof}

\begin{remark}
Note that the statement of Theorem \ref{thm:circle_upper} is reminiscent of arguably the most popular technique for estimating the numerical range of a matrix, which consists of choosing angles $\theta_1, \ldots, \theta_\ell$, and computing the largest and smallest eigenpairs of $(e^{i\theta_j} A + e^{- i \theta_j} A^*)/2$. This algorithm was proposed by Johnson \cite{johnson1978numerical} and is implemented in the $\mathrm{fv}$ function in the Matrix Computation Toolbox \cite{higham2002matrix} and the $\mathrm{fov}$ function in Chebfun \cite{driscoll2014chebfun}, using the QR algorithm to estimate extreme eigenvalues. In addition, Braconnier and Higham created a variant using a version of the Lanczos method, which exploits the fact that only the largest and smallest eigenpairs are needed \cite{braconnier1996computing}. As mentioned in the introduction, using a Krylov subspace of dimension $m$, one could guarantee an approximation of the form $\tfrac{\ln^2 n}{m^2}$ for each angle. When $\Lambda(A) \subset \mathbb{S}^0$, Theorem \ref{thm:circle_upper} proves that computing the numerical range of $H_m$ instead provides this same approximation quality for every angle simultaneously (i.e., without having to build a different Krylov subspace for each angle).
\end{remark}

We note that the Hausdorff distance for a projection onto any line also immediately implies the same bound for the Hausdorff distance in the plane.

\begin{corollary}\label{cor:circle_upper}
Let $n \in \mathbb{N}$, and $A \in \mathbb{C}^{n \times n}$ be normal with its spectrum on $\mathbb{S}^0$. Then
\[\Prob_{\bm b \sim \mathrm{Unif}(\mathbb{S}^{n-1})} \bigg[d_{H}\big(W(H_{2m+1}), W(A) \big) \le \frac{(2+\alpha)^2  \ln^2 n }{16 m^2} \bigg] \ge 1 - \frac{64 e^2 m^2}{n^{\alpha}}. \]
\end{corollary}

\begin{proof}
By Proposition \ref{prop:polytopevertices} and the inclusion $W(H_{2m+1}) \subset W(A)$, it suffices to show that, for any extreme point $\lambda$ of $\mathrm{conv}(\Lambda(A))$, the desired bound holds. For an extreme point $\lambda \in \partial W(A)$, consider the line through $\lambda$ and the closest point in $W(H_{2m+1})$. By the convexity of $W(H_{2m+1})$, the distance between $\lambda$ and $W(H_{2m+1})$ is equal to the distance between the projection of both onto this line (this follows from the standard proof of the hyperplane separation theorem, see \cite[pg. 11]{grunbaum1967convex}). Applying Theorem \ref{thm:circle_upper} to this line completes the proof.
\end{proof}

\begin{theorem}\label{thm:circle_lower}
Let $n = k \, \ell $, $k, \ell, m \in \mathbb{N}$, $k \ge m \ge 10$, $\ell \ge 8\, m^4  \ln k $, $\bm 1 = (1,\ldots,1)^T \in \mathbb{C}^{\ell}$, $\omega = \exp\big(\frac{2 \pi i}{k}\big)$, $\bm \theta = (\omega,\omega^2,\ldots,\omega^k)^T \in \mathbb{C}^{k}$, and $A = \mathrm{diag}(\bm \theta \otimes \bm 1) \in \mathbb{C}^{n\times n}$. If $\bm b \sim \mathrm{Unif}(\mathbb{S}^{n-1})$, then
\[ \Prob\bigg[ d\big(W(H_{m}), \partial W(A) \big) \ge \frac{2}{m^2}  \bigg] \ge 1 - \frac{2}{k} . \]
\end{theorem}

\begin{proof}
The proof technique here is nearly identical to that of Theorem \ref{thm:norm_lower}.
Without loss of generality, let $\bm{b} \sim \compnormvec$ be a complex Gaussian vector. Let $\hat A = \mathrm{diag}( \bm \theta) \in \mathbb{C}^{k \times k}$, $ \hat {\bm b} \in \mathbb{R}^{k}$ be such that 
\[\hat {\bm b}_j= \sqrt{\sum_{p = 1}^{\ell} |\bm b_{p+(j-1)\ell}|^2},\]
and $\hat H_m$ be the orthogonal projection of $\hat A$ onto $\mathcal{K}_{m}(\hat A,\hat {\bm b})$ ($\mathrm{dim}(\mathcal{K}_{m}(\hat A,\hat {\bm b})) = m$ with probability one). Again, by Equation \ref{eqn:normal_RQ}, $H_m$  and $\hat H_m$ have the same numerical range, and so we may instead work with $\hat H_m$ and note that two times the squared entries of $\hat {\bm b}$ are independent chi-squared variables with $2\ell$ degrees of freedom each. Let $\tilde {\bm b} = \sqrt{\ell}\,  \bm{1}\in \mathbb{C}^{k}$ and $\tilde H_m = \tilde Q_m^* \hat A \tilde Q_m$ be the orthogonal projection of $\hat A$ onto $\mathcal{K}_{m}(\hat A,\tilde {\bm b})$. The vectors $\hat A^p \tilde {\bm b}$ are orthogonal, with constant norm
  \[    \langle \hat A^p \tilde {\bm b}, \hat A^q \tilde {\bm b}\rangle= \ell \sum_{j=1}^{k} \bm{\theta}_j^{(p-q)}= \ell \sum_{j=1}^k e^{2 \pi i(p-q)\frac{j}{k}}  = \begin{cases}
             k \ell &  \text{ if } p=q \\
            0 & \text{ otherwise },
        \end{cases}\]
and so $[\tilde H_m]_{p+1,p} = 1$ and equals zero elsewhere. This matrix $H_m$ has a well-known numerical range, with radius $\cos(\pi/(m+1))$ \cite{yuanwu1998numerical}. Now, consider how different the situation can be for $\hat H_m$. For $\hat {\bm b}$, the vectors $\hat A^p \hat {\bm b}$, $p = 0,\ldots,m-1$, have the same span as $\sqrt{\ell}(I+D) \tilde Q_m$, where $D$ is diagonal with $D_{pp} = \frac{\hat  {\bm{b}}_p}{\sqrt{\ell}} - 1$.  By Proposition \ref{prop:chisq} with $t:= m^{-2}$ and noting that $\ell \ge 8\, m^4  \ln k $, 
\[\mathbb{P}\left[ \max_{j =1,\ldots,k} |2 \hat{\bm b}_j^2 - 2\ell| \ge \frac{2 \ell}{m^2}  \right] \le 2 k \exp \bigg( - \frac{2 \ell}{8 m^4} \bigg) \le \frac{2}{k}.\]
Therefore, $\|D\|_2 \le 1-  \sqrt{1 - m^{-2}}$
with probability at least $1 -2/k$. Again, similar to the proof of Theorem \ref{thm:norm_lower}, by Proposition \ref{prop:polar} applied to $\tilde Q_m$ and $(I+D)\tilde Q_m$, 
\[\|\tilde H_m - U^* \hat H_m U\|_2 \le \frac{4\left(1-  \sqrt{1 - m^{-2}}\right)}{\sqrt{1 - m^{-2}}} , \]
for some unitary matrix $U \in \mathbb{C}^{m \times m}$. Therefore, the largest modulus of an element of $W(\hat H_m)$ is at most 
\[f(m):=\cos\left(\frac{\pi}{m+1}\right) + \frac{4\left(1-  \sqrt{1 - m^{-2}}\right)}{\sqrt{1 - m^{-2}}}.\]
The function $(1-f(m))m^2$ is monotonically increasing, and so $(1-f(m))m^2 \ge 100(1-f(10))>2$ for all $m \ge 10$. Therefore, $f(m) \le 1 - 2/m^2$ for $m \ge 10$.
\end{proof}

Again, similar to the lower bound of Theorem \ref{thm:norm_lower} (see Remark \ref{rem:lower}), the requirements on $n$ relative to $m$ are quite pessimistic. See Figure \ref{fig:lower} to observe this behavior numerically for relatively small $n$.

\section{The Numerical Range and Extreme Eigenvalues of a Non-Normal Matrix}\label{sec:nonnormal}

To this point, we have produced estimates for the approximation quality of $W(H_m)$ for $W(A)$ for normal matrices through estimates for extreme eigenvalues of $A$. These estimates crucially rely on the fact that, for a normal matrix, the numerical range equals the convex hull of the eigenvalues $W(A) = \mathrm{conv}( \Lambda(A))$. When $A$ is non-normal, $W(A)$ may be strictly larger than $\mathrm{conv}(\Lambda(A))$. The extent to which this can occur can be easily quantified in terms of the conditioning of $V$ in the eigendecomposition $A = V \Lambda V^{-1}$.

\begin{lemma}\label{lm:convtonum}
Let $A = V \Lambda V^{-1}$, with $\Lambda$ diagonal and $\|\Lambda\|_2 \le 1$. Then
\[ d_H( \mathrm{conv}( \Lambda(A) ), W(A)) \le \|V \|_2 \| (V^* V)^{-1/2} - I\|_2 + \| (V^* V)^{1/2} - I\|_2.\]
\end{lemma}

\begin{proof}
In order to bound the Hausdorff distance, it suffices to define $\hat A = V(V^*V)^{-1/2} \Lambda (V^*V)^{1/2} V^{-1}$, note that $W(\hat A) = \mathrm{conv}( \Lambda(A) )$, and bound the difference in Rayleigh quotient (itself bounded by the two-norm) between $A$ and $\hat A$ for an arbitrary vector $\bm x$. We have
\begin{align*}
\|A - \hat A \|_2 &\le \|V \Lambda V^{-1} - V(V^*V)^{-1/2}\Lambda (V^*V)^{1/2} V^{-1}\|_2 \\
&\le \|(V^*V)^{1/2}\Lambda(V^*V)^{-1/2} - \Lambda \|_2 \\
&\le \|(V^*V)^{1/2}\Lambda(V^*V)^{-1/2} - (V^*V)^{1/2}\Lambda\|_2 + \|(V^*V)^{1/2}\Lambda - \Lambda \|_2 \\
&\le \|(V^*V)^{1/2} \|_2 \| (V^* V)^{-1/2} - I\|_2 + \| (V^* V)^{1/2} - I\|_2.
\end{align*}
\end{proof}

When this distance is large, the techniques used in Section \ref{sec:normal} are no longer sufficient to provide a quantitative estimate on $W(A)$ using $W(H_m)$. Unfortunately, in the worst-case, this issue is inherent to the problem. The following example illustrates that the bound in Lemma \ref{lm:convtonum} is essentially tight in the worst case, and that $W(H_m)$ in some cases no loner provides a reasonable estimate of $W(A)$ for even moderately sized $m$.

\begin{example}\label{ex:bad}
Consider the matrix $A = \Lambda + \gamma \bm{e}_1 \bm{e}^T_n \in \mathbb{R}^{n \times n}$, where $\gamma>0$ and $\Lambda$ is a diagonal matrix with $\Lambda_i = \cos \big(\frac{(i-1)\pi}{n-1}\big)$ for $i = 1,\ldots,n$. $A$ has all its eigenvalues in $[-1,1]$, numerical range 
\[W(A) = \left\{ z \in \mathbb{C} \, \bigg| \, \frac{\mathrm{Re}(z)^2}{1 + \gamma^2/4} + \frac{\mathrm{Im}(z)^2}{\gamma^2/4} \le 1 \right\},\]
and eigenvector condition number $\kappa_{V}(A) = \sqrt{1+\gamma^2/4}$. As $\gamma$ increases, the distance between $\partial W(A)$ and $[-1,1]$ and the size of $\kappa_V(A)$ both increase linearly. The numerical range of most Krylov subspaces will contain $[-1,1]$ (the convex hull of the eigenvalues of $A$) relatively quickly, but will fail to approximate $W(A)$ in a reasonable amount of time. See Figure \ref{fig:nonnormal} for an example with $n = 1000$, $\gamma = 2$, and $m = 100$.
\end{example}

In this section, we present two main results. The first provides estimates on the Hausdorff distance between $W(H_m)$ and $W(A)$ (Theorem \ref{thm:non-normal} in Subsection \ref{sub:nearnormal}). Unsurprisingly, in light of Example \ref{ex:bad}, these estimates degrade rather quickly once $\kappa(V)$ is larger than $1 +o_n(1)$. However, even though $W(H_m)$ may no longer approximate $W(A)$ well once $\kappa(V)$ grows, as long as the eigenvectors still possess repulsive behavior (quantified below in Subsection \ref{sub:repulse}), precise statements regarding the approximation that $W(H_m)$ provides to $\mathrm{conv}(\Lambda(A))$ can still be made, in some cases, for $\kappa(V)$ as large as $n^{1/2-\epsilon}$ (Theorem \ref{thm:eigenhull} in Subsection \ref{sub:repulse}). 

\subsection{Estimating the Numerical Range for a Nearly Normal Matrix}\label{sub:nearnormal} First, we estimate the approximation the $W(H_m)$ provides to an arbitrary extreme eigenvalue $\lambda$ of $\Lambda(A)$. In the previous section, we produced bounds on this approximation by using extremal polynomials from Subsection \ref{sub:poly} applied to Equation \ref{eqn:normal_RQ}. Here, we apply the same techniques, but must account for two key differences between Equation \ref{eqn:nonnormal_RQ} and Equation \ref{eqn:normal_RQ}: the existence of the term $V^* V$ in both the numerator and denominator, and that $\bm{\alpha} = V^{-1} \bm{b}$ is no longer uniform on $\mathbb{S}^{n-1}$. The former issue is treated by the following lemma and the latter issue is treated using Proposition \ref{prop:perturbed_basis}.

\begin{lemma}\label{lm:non-normal_worstcase}
Let $V$ be invertible. Then
\[\left|\frac{\bm x^* V^* V D \bm{x}}{\bm x^* V^* V \bm x}\right| \le \|V^{-1}\|_2^2 \left| \frac{\bm x^* D \bm{x}}{\bm x^*  \bm x} \right| + \|D\|_2 \|V^{-1}\|_2 \|V - (V^{-1})^*\|_2 .\]
\end{lemma}

\begin{proof}
We have 
\begin{align*}
\left|\frac{\bm x^* V^* V D \bm{x}}{\bm x^* V^* V \bm x}\right| &\le \left| \frac{\bm x^* D \bm{x}}{\bm x^* V^* V  \bm x} \right| + \left|\frac{\bm x^* (V^*V -I)D   \bm{x}}{\bm x^* V^* V \bm x}\right| \\
&= \left| \frac{\bm x^* D \bm{x}}{\bm x^*  \bm x}\right| \left| \frac{\bm x^* \bm x}{\bm x^* V^* V  \bm x}\right| + \left|\frac{\bm x^* V^*( V -(V^{-1})^*)D V^{-1} V  \bm{x}}{\bm x^* V^* V \bm x}\right| \\
&\le \left| \frac{\bm x^* D \bm{x}}{\bm x^*  \bm x} \right| \|V^{-1}\|_2^2 +  \|V - (V^{-1})^*\|_2 \|D\|_2 \|V^{-1}\|_2.
\end{align*}
\end{proof}

A fair amount is lost in the bound of Lemma \ref{lm:non-normal_worstcase} compared to what one may expect in practice. For instance, this bound handles the worst case scenario for a vector $\bm{x}$, but fails to exploit the potential benefits of using a random $\bm{b}$. However, this lemma is sufficient for the purposes of estimating $W(A)$, as when $\kappa_V(A)$ is larger than $1 + o_n(1)$, there can be other barriers to accurate approximation (see Example \ref{ex:bad}).

\begin{figure}
\begin{center}
\subfigure[$W(A)$ and $W(H_{100})$ for ten different $\bm{b} \sim \mathrm{Unif}(\mathbb{S}^{n-1})$]{\includegraphics[height = 2.1in]{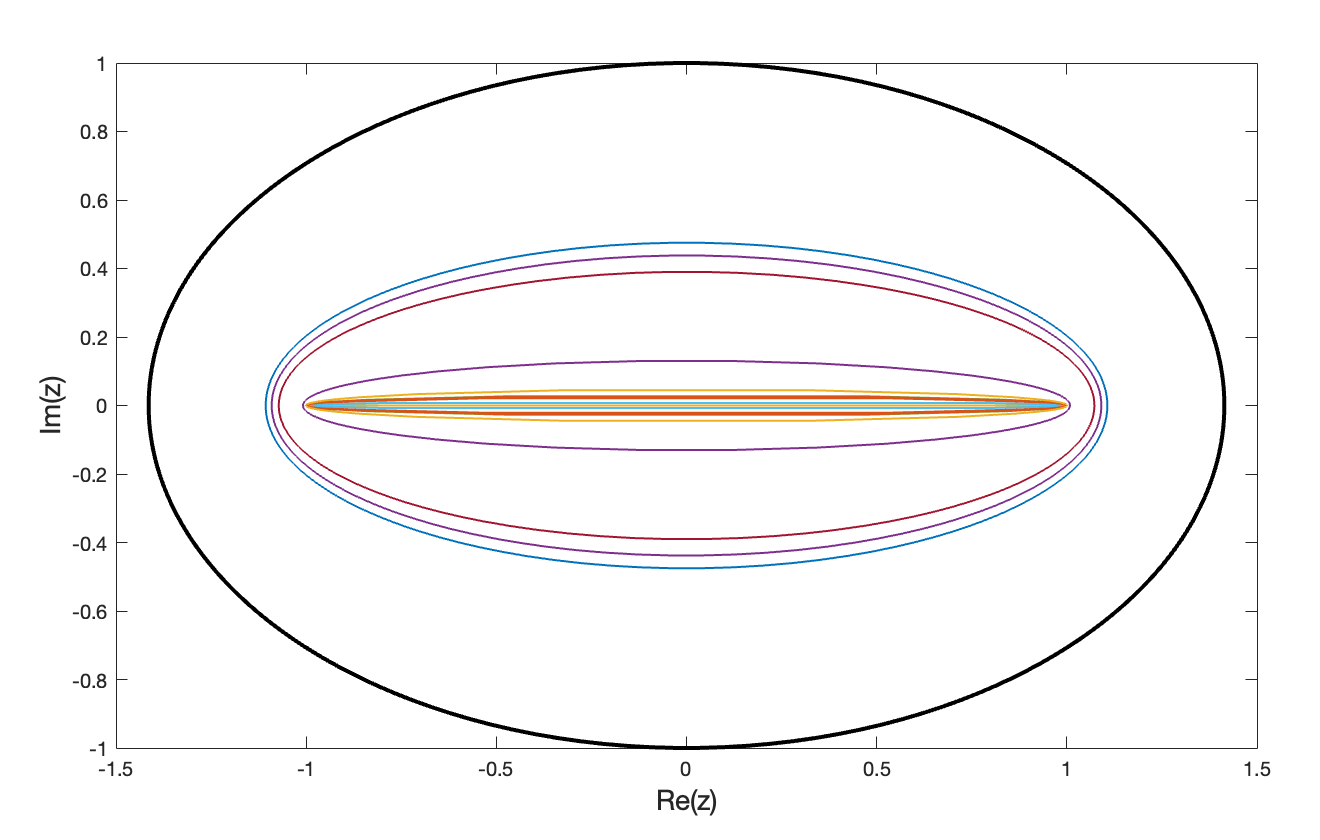}} 
\subfigure[$d_H(W(A),W(H_m))$ for $m = 1,\ldots,100$ and ten different $\bm{b} \sim \mathrm{Unif}(\mathbb{S}^{n-1})$]{\includegraphics[height = 2.1in]{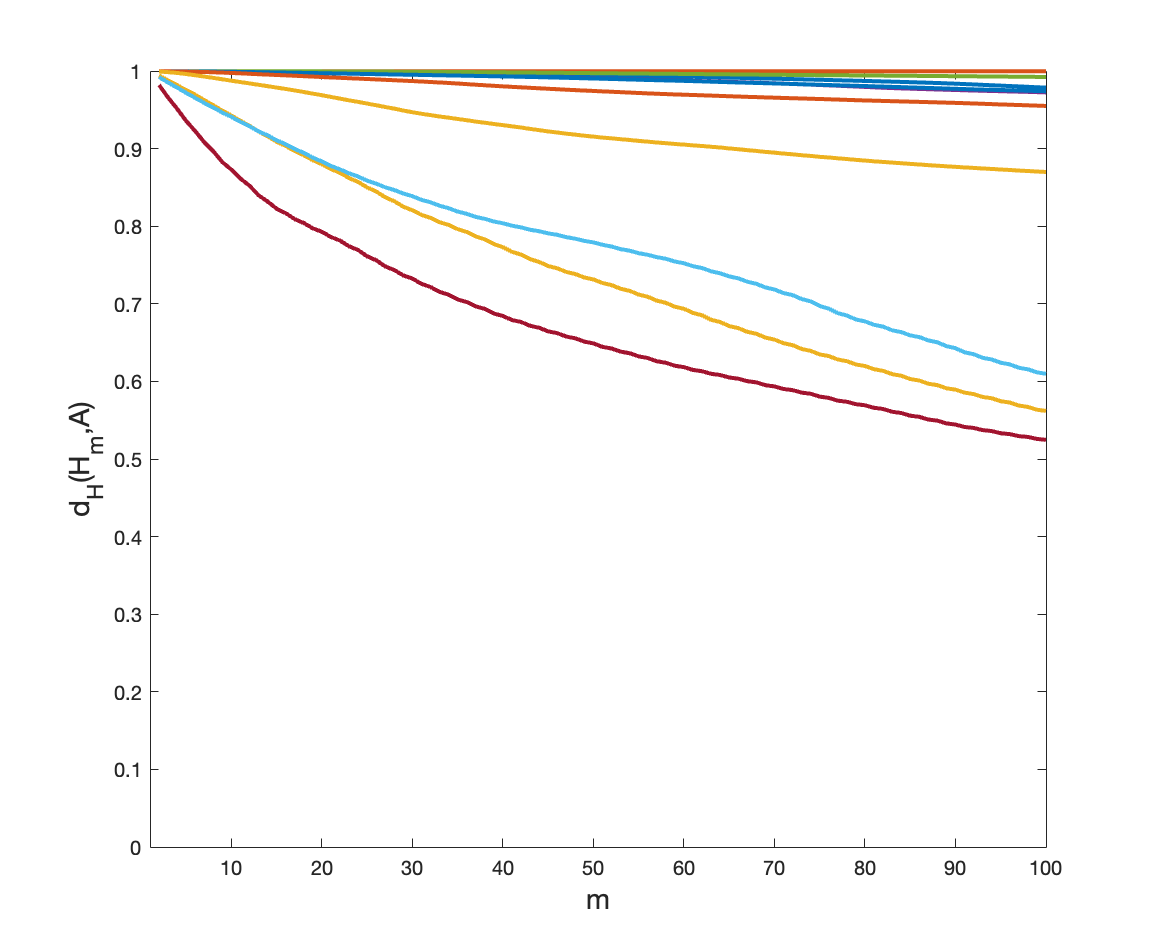}}
\caption{The numerical range of $H_{m}$ for $A = \mathrm{diag}(\cos \tfrac{(j-1)\pi}{n}) + 2 \bm{e}_1 \bm{e}_n^T$ and ten different $\bm{b} \sim \mathrm{Unif}(\mathbb{S}^{n-1})$, where $n = 1000$ and $m = 100$. The numerical range quickly (approximately) contains $[-1,1]$ (the convex hull of eigenvalues of $A$) in all instances, but fails to reasonably approximate the numerical range of $A$. In over half the cases, $W(H_{100})$ is barely bigger than $[-1,1]$, and in all cases the Hausdorff distance between $W(A)$ and $W(H_{100})$ is still at least $1/2$.}
\label{fig:nonnormal}
\end{center}
\end{figure}

Combining Lemmas \ref{lm:convtonum} and \ref{lm:non-normal_worstcase} and applying the same techniques used in the proofs of Lemma \ref{lm:single_eig} and Theorem \ref{thm:circle_upper}, we obtain the following result.

\begin{theorem}\label{thm:non-normal}
Let $m, n \in \mathbb{N}$, and $A \in \mathbb{C}^{n \times n}$ be a diagonalizable matrix with eigendecomposition $A = V \Lambda V^{-1}$, where the columns of $V$ have norm one. Then
\[d_H\big(W(H_{6m+1}),W(A) \big) \le \bigg(\frac{6\kappa^2(V)  }{m}\ln  \bigg( \frac{e m \|\bm \alpha\|^2_2 }{6 \min_{j \in [n]} | \bm \alpha_j |^2}\bigg) +4\kappa(V) (\kappa(V)-1)\bigg) \mathrm{diam}(\mathrm{conv}(\Lambda(A))).\]
In particular, if $\bm b \sim \mathrm{Unif}(\mathbb{S}^{n-1})$, then
\[\Prob\bigg[ \frac{d_{H}\big(W(H_{6m+1}),W(A) \big)}{\mathrm{diam}(\mathrm{conv}(\Lambda(A)))} \le \frac{6\kappa^2(V)  }{m}\ln  \big( n^{2+\alpha} \kappa^2(V)\big) +4\kappa(V) (\kappa(V)-1)  \bigg] \ge 1 - \frac{5m}{4n^{\alpha}}. \]
\end{theorem}

\begin{proof}
The proof of this result is quite similar to that of Lemma \ref{lm:single_eig} and Theorem \ref{thm:main}, with the additional use of Proposition \ref{prop:perturbed_basis} and Lemmas \ref{lm:convtonum} and \ref{lm:non-normal_worstcase}. First, we estimate the distance between an arbitrary extreme point $\lambda \in \Lambda(A)$ and $W(H_{6m+1})$.

Without loss of generality, let $\lambda = 0$ be the extreme point, $\mathrm{diam}(W(A)) =1$, and $W(A) \subset D_0$, the half disk. By Equation \ref{eqn:nonnormal_RQ} and Lemma \ref{lm:non-normal_worstcase},
\[\min_{x \in \mathcal{K}_{6m+1}(A,b)} \bigg| \frac{\bm x^* A \bm x}{\bm x^* \bm x} \bigg| \le \|V^{-1}\|_2^2 \min_{p \in \mathcal{P}_{6m}} \left| \frac{\bm \alpha^* p(\Lambda)^* \Lambda p(\Lambda) \bm{\alpha}}{\bm \alpha^* p(\Lambda)^* p(\Lambda)  \bm \alpha} \right|  +  \|V - (V^{-1})^*\|_2  \|V^{-1}\|_2.\]
We have $\|V^{-1}\|_2 \le \kappa(V)$ and 
\[\|V - (V^{-1})^*\|_2 \le \max_{j \in [n]} \sigma_j(V) - \sigma_j(V)^{-1} \le 2 (\kappa(V)-I).\]
Let $\delta >0$ and $S_\delta = \{\Lambda_j \in \Lambda(A) \, | \, |\Lambda_j| > \delta \}$. Then
\begin{align*}
\min_{p \in \mathcal{P}_{6m}} \left| \frac{\bm \alpha^* p(\Lambda)^* \Lambda p(\Lambda) \bm{\alpha}}{\bm \alpha^* p(\Lambda)^* p(\Lambda)  \bm \alpha} \right| &\le   \min_{p \in \mathcal{P}_{6m}} \left| \frac{\sum_{\Lambda_j \in S_\delta} |\bm \alpha_j|^2  |p(\Lambda_j)|^2 \Lambda_j}{\sum_{j=1}^n |\bm \alpha_j|^2  |p(\Lambda_j)|^2} \right| +  \left| \frac{\sum_{\Lambda_j \not \in S_\delta} |\bm \alpha_j|^2  |p(\Lambda_j)|^2 \Lambda_j}{\sum_{j=1}^n |\bm \alpha_j|^2  |p(\Lambda_j)|^2} \right| \\
&\le \min_{p \in \mathcal{P}_{6m}} \frac{ \sum_{\Lambda_j \in S_\delta} |\bm \alpha_j|^2  |p(\Lambda_j)|^2 }{\min_{j \in [n]} | \bm \alpha_j |^2  |p(0)|^2} + \delta.
\end{align*}
Now, consider the choice $p(z) = \widetilde P_{m, \delta}(z)$, as defined in Proposition \ref{prop:remez}. By Inequality \ref{ineq:remez},
\[\frac{\sum_{\Lambda_j \in S_\delta} |\bm \alpha_j|^2  |\widetilde P_{m,\delta}(\Lambda_j)|^2 }{\min_{j \in [n]} | \bm \alpha_j |^2  |\widetilde P_{m,\delta}(0)|^2} \le \frac{\|\bm \alpha\|^2_2}{\min_{j \in [n]} | \bm \alpha_j |^2} \exp\left(-\frac{m \delta}{6} \right) .\]
 Setting $\delta = 6 m^{-1} \ln  \big( \frac{m \| \bm \alpha \|^2_2 }{6 \min_{j \in [n]} | \bm \alpha_j |^2}\big)$ (if $\delta >1$, the result is trivially true), we have
 \[\min_{p \in \mathcal{P}_{6m}} \left| \frac{\bm \alpha^* p(\Lambda)^* \Lambda p(\Lambda) \bm{\alpha}}{\bm \alpha^* p(\Lambda)^* p(\Lambda)  \bm \alpha} \right| \le  \frac{\|\bm \alpha\|^2_2}{\min_{j \in [n]} | \bm \alpha_j |^2} \exp\left(-\frac{m \delta}{6} \right) + \delta \le \frac{6}{m}\left(1 +   \ln  \bigg( \frac{m \|\bm \alpha\|^2_2 }{6 \min_{j \in [n]} | \bm \alpha_j |^2}\bigg) \right) .\]
Altogether, this implies that
\[\min_{x \in \mathcal{K}_{6m+1}(A,b)} \bigg| \frac{\bm x^* A \bm x}{\bm x^* \bm x} \bigg| \le  \frac{ 6 \kappa^2(V)  }{m}\ln  \bigg( \frac{e m \|\bm \alpha\|^2_2 }{6 \min_{j \in [n]} | \bm \alpha_j |^2}\bigg) +2\kappa(V) (\kappa(V)-1).\]
By Proposition \ref{prop:polytopevertices}, a bound for all extreme eigenvalues of $A$ immediately implies a bound for the Hausdorff distance between $\mathrm{conv}(\Lambda(A))$ and $W(H_{6m+1})$. Combining this with Lemma \ref{lm:convtonum}, we obtain
\begin{align*}
    d_H(W(H_{6m+1}),W(A)) &\le d_H(W(H_{6m+1}),\mathrm{conv}(\Lambda(A))) + d_H(\mathrm{conv}(\Lambda(A)),W(A)) \\
    &\le \frac{6 \kappa^2(V)  }{m}\ln  \bigg( \frac{e m \|\bm \alpha\|^2_2 }{6 \min_{j \in [n]} | \bm \alpha_j |^2}\bigg) +4\kappa(V) (\kappa(V)-1).
\end{align*}
Now, suppose that $\bm b \sim \mathrm{Unif}(S^{n-1})$. By Proposition \ref{prop:perturbed_basis} with $M = V^{-1}$, $t = e m /(6 n^{\alpha})$, we have
\[\Prob\left[ \frac{\min_{j \in [n]} |[V^{-1} \bm b]_j|^2}{\|V^{-1} \bm b\|_2} \ge \frac{e m }{6 n^{2+\alpha} \kappa^2(V)} \right] \ge 1 - \frac{e^2 m}{6 n^{\alpha}} \ge 1- \frac{5m}{4n^\alpha} ,\]
completing the proof.
\end{proof}

An analogous statement to Theorem \ref{thm:non-normal} for $\Lambda(A) \subset \mathbb{S}^0$ can also be made by combining the above proof with the proofs of Theorem \ref{thm:circle_upper} and Corollary \ref{cor:circle_upper}.

\begin{remark}
In Figure \ref{fig:nonnormal}, despite a constant condition number, $W(H_m)$ approximately encloses $\mathrm{conv}(\Lambda(A))$. This can be easily explained in the following way. When restricted to the subspace of $\mathbb{C}^{n \times n}$ with first and last component equal to zero, $A$ is normal, and, when $m >2$, $\mathcal{K}_m(A,\bm{b})$ contains a vector in this subspace. Therefore, $W(H_m)$ must provide a very good estimate of $\mathrm{conv}\big( \Lambda(A) \backslash\{\Lambda_1,\Lambda_n\} \big)$, which in this case is nearly identical to $\mathrm{conv}(\Lambda(A))$. In Subsection \ref{sub:repulse}, we consider a general class of matrices whose eigenvector condition number is of moderate size, yet still $W(H_m)$ approximately encloses $\mathrm{conv}(\Lambda(A))$, due to relatively well-behaved eigenvectors.
\end{remark}

\subsection{Bounding the Eigenvalues of a Non-Normal Matrix with Repulsive Eigenvectors}\label{sub:repulse}
While Subsection $\ref{sub:nearnormal}$ illustrates that estimates of the numerical range break down as the eigenvector condition number increases, we can still obtain a reasonable estimate for the convex hull of eigenvalues when the eigenvectors are relatively well-behaved. Consider the following example:
\begin{example}\label{ex:eigenangles}
    Let $A = V\Lambda V^{-1} \in \mathbb{C}^{n \times n}$, with $V^*V = (1-n^{-2/3})I + n^{-2/3} \bm{1} \bm{1}^T$ and $\Lambda$ as defined in Theorem \ref{thm:norm_lower}. Figure \ref{fig:nonnormal-angles} demonstrates that $\mathrm{conv}(\Lambda(A))$ can be approximated with order $1/m$ convergence, as the condition $|[V^*V]_{jk}| \le n^{-2/3}$ prevents the eigenvectors from clustering.
\end{example}

The following definition parameterizes our notion of eigenvector repulsion, preventing eigenvectors from adversarially combining in any single direction. 

\begin{definition}\label{def:betanormal}
We say a matrix $V \in \mathbb{C}^{n \times n}$ is $\beta$-normal for some $\beta >0$ if its columns have norm one and, for $V = [ \bm{v}_1 \, \ldots \, \bm{v}_n]$ and $(V^{-1})^* = [ \bm{w}_1 \, \ldots \, \bm{w}_n ]$,
\begin{equation}\label{ineq:beta}
\sum_{\ell \in [n] \backslash k } |\bm v_k^* \bm v_\ell| (|\bm w_\ell^* \bm w_j| + |\bm w_k^* \bm w_j|)\leq n^{-\beta}\|\bm w_j\|^2 \qquad \text{for all} \quad j,k \in [n].
\end{equation}
\end{definition}

When this $\beta$-normal condition holds, we can effectively treat the additional difficulty introduced by the term $V^*V$ in the numerator and denominator of Equation \ref{eqn:nonnormal_RQ}.

\begin{figure}
\begin{center}
\subfigure[$W(A)$, $\Lambda(A)$ and $W(H_{30})$ for ten different $\bm{b} \sim \mathrm{Unif}(\mathbb{S}^{n-1})$]{\includegraphics[height = 1.6in]{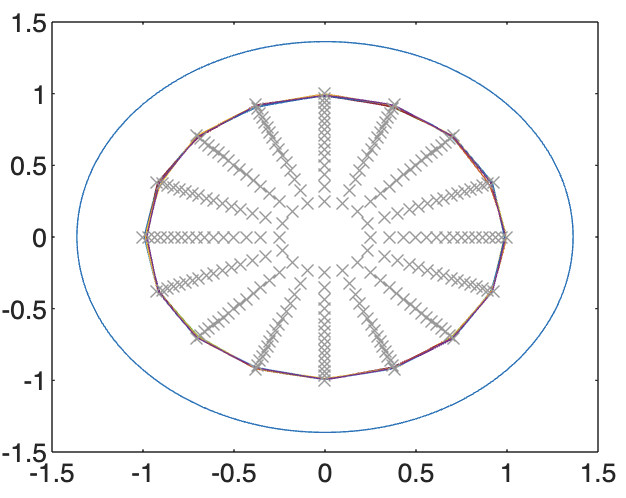}} 
\subfigure[$\tilde{d}_H(\mathrm{conv}(\Lambda(A)), W(H_m))$ for \break $m = 1,\ldots,50$ and ten different \break $\bm{b} \sim \mathrm{Unif}(\mathbb{S}^{n-1})$]{\includegraphics[height = 1.6in]{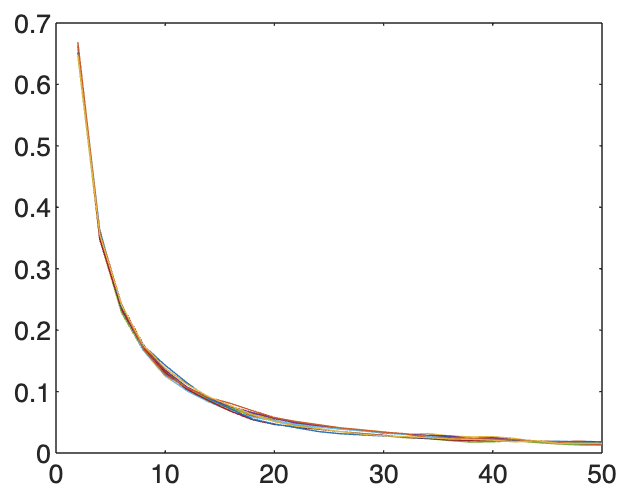}}
\subfigure[$m \times \tilde{d}_H(\mathrm{conv}(\Lambda(A)),W(H_m))$ for \break $m = 1,\ldots,50$ and ten different \break$\bm{b} \sim \mathrm{Unif}(\mathbb{S}^{n-1})$]{\includegraphics[height = 1.6in]{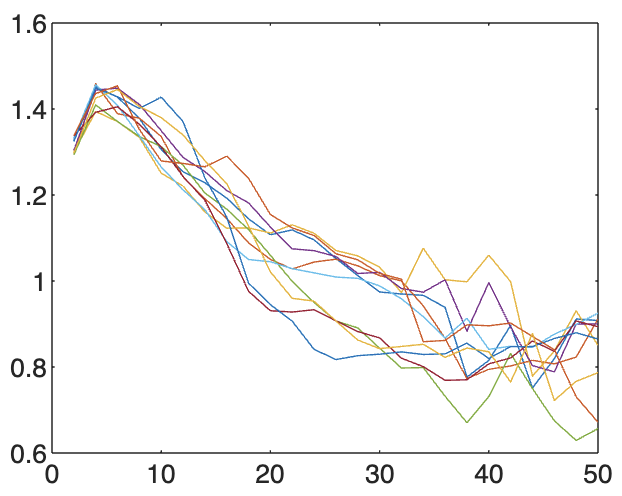}}
\caption{$W(H_{m})$ for $A$ in Example \ref{ex:eigenangles}. Here we let $n=2560$, $m=16$, and $l=10$.}
\label{fig:nonnormal-angles}
\end{center}
\end{figure}

\begin{lemma}\label{lem:VVSmall}
     Let $n \ge 2$ and $\beta >2 \ln(2\ln(e n))/\ln(n)$, $V \in \mathbb{C}^{n \times n}$ be $\beta$-normal (see Definition \ref{def:betanormal}), and $D,P \in \mathbb{C}^{n \times n}$ be diagonal with $\|D\|_2 \le 1$. Then, for $\bm b \sim \mathrm{Unif}(\mathcal{S}^{n-1}(\mathbb{C}))$ and $\varepsilon = 2n^{-\beta/2}\ln(e n)$,
     \[\prb{\bigg|\frac{[PV^{-1}\bm b]^*V^*VD[PV^{-1}\bm b]}{[PV^{-1}\bm b]^*V^*V[PV^{-1}\bm b]}\bigg| \geq \left(\frac{1}{1-\varepsilon}\right)\left(\bigg|\frac{[PV^{-1}\bm b]^*D[PV^{-1}\bm b]}{[PV^{-1}\bm b]^*[PV^{-1}\bm b]}\bigg| + \varepsilon \right)} \leq e n^{-\beta/2} + 2 n^{-1/2} + 4 n^{-1/4}.\]
\end{lemma}
\begin{proof}
    For simplicity, let $W = (V^{-1})^*$, and, without loss of generality, suppose that $|P_j| \le |P_k|$ for $j <k$. We first treat the denominator $\|VPW^*\bm b\|^2$; the numerator will be bounded in a similar manner. For the desired denominator $\|PW^*\bm b\|^2$, let $\phi := \mathbb{E}[\|PW^*\bm b\|^2]$ denote its expectation. Then, by Proposition \ref{prop:traces},
    \[\phi = \mathbb{E}[\|PW^*\bm b\|^2] = \mathrm{trace}(WP^*PW^*) = \mathrm{trace}(P^*PW^*W) = \sum_{j=1}^n |P_j|^2\|\bm w_j\|^2.\]
The expectation of $\|VPW^*\bm b\|^2$ is given by
\[\mathbb{E}[\|VPW^*\bm b\|^2] = \mathrm{trace}(WP^*V^*VPW^*) = \mathrm{trace}(V^*VPW^*WP^*) = \sum_{j=1}^n \sum_{k=1}^n \bm v_j^*\bm v_k  P_k  \bm w_k^*\bm w_j  \overline{P_j}.\]
By the ordering $|P_j| \le |P_k|$ for $j <k$ and Inequality \ref{ineq:beta} (with $j = k$),
\begin{align*} \mathbb{E}[\|VPW^*\bm b\|^2] &= \sum_{j=1}^n |P_j|^2 \|\bm w_j\|^2 + 2 \sum_{j = 1}^n \sum_{k = 1}^{j-1} \mathrm{Re}(  P_k \overline{P_j} \, \bm v_j^*\bm v_k  \, \bm w_k^*\bm w_j  )  \\
    &\geq \sum_{j=1}^n |P_j|^2 \|\bm w_j\|^2 - 2 \sum_{j = 1}^n |P_j|^2 \sum_{k = 1}^{j-1} |\bm v_j^*\bm v_k|   |\bm w_k^*\bm w_j| \\
    &\geq (1-2n^{-\beta})\phi.
    \end{align*}
Now, let $M = WP^*(V^*V - I)PW^*$, and consider the error term $\|VPW^*\bm b\|^2 - \|PW^*\bm b\|^2 = \bm{b}^* M \bm{b}$. By Proposition \ref{prop:traces}, its variance is given by
    \begin{align*}
       \|M\|_F^2 &=  \mathrm{trace}(([WP^*][V^*V - I][PW^*])^2) \\
        &= \mathrm{trace}(([V^*V - I][PW^*][WP^*])^2) \\
        &= \sum_{j=1}^n \sum_{k=1}^n \Bigg(\sum_{\ell \in [n] \backslash j} \bm v_j^*\bm v_\ell \, P_\ell  \overline{P_k} \, \bm w_\ell^*\bm w_k \Bigg) \Bigg(\sum_{m \in [n] \backslash k}  \bm v_k^*\bm v_m \, P_m \overline{P_j} \, \bm w_m^*\bm w_j  \Bigg) .
    \end{align*}
    Rearranging the index pairs $j,\ell$ and $k,m$, and applying Inequality \ref{ineq:beta}, we obtain
    \begin{align*}
\|M\|_F^2 &= \sum_{j=1}^n \sum_{k=1}^n \Bigg(\sum_{\ell =1}^{j-1} \bm v_j^*\bm v_\ell \, P_\ell  \overline{P_k} \, \bm w_\ell^*\bm w_k 
+ \bm v_\ell^*\bm v_j \, P_j  \overline{P_k} \, \bm w_j^*\bm w_k 
\Bigg) \\
&\qquad \qquad \quad \cdot \Bigg(\sum_{m =1}^{k-1}  \bm v_k^*\bm v_m \, P_m \overline{P_j} \, \bm w_m^*\bm w_j +  \bm v_m^*\bm v_k \, P_k \overline{P_j} \, \bm w_k^*\bm w_j \Bigg) \\
&\le \sum_{j=1}^n \sum_{k=1}^n |P_j|^2 |P_k|^2   \Bigg(\sum_{\ell =1}^{j-1} |\bm v_j^*\bm v_\ell| \big( |\bm w_\ell^*\bm w_k | + |\bm w_j^*\bm w_k | \big) 
\Bigg) \Bigg(\sum_{m =1}^{k-1}  |\bm v_k^*\bm v_m| \big(  | \bm w_m^*\bm w_j| +  | \bm w_k^*\bm w_j|\big) \Bigg)  \\
&\le n^{-2\beta} \sum_{j=1}^n \sum_{k=1}^n |P_j|^2 |P_k|^2 \|\bm{w}_k\|^2 \| \bm{w}_j \|^2 \\
&\le n^{-2\beta} \phi^2.
\end{align*}

Now, let $X = \|PW^*\bm b\|^2$, $Y = \|VPW^*\bm b\|^2$, and $t \ge 1$. By Proposition \ref{cor:subexp} (using the inequality $\|M\|_2 \leq \|M\|_F$) and Proposition \ref{prop:notsmall},
    \begin{align*}\prb{Y \leq (1-2n^{-\beta/2}t)X} &= \prb{X-Y \geq 2n^{-\beta/2}tX}  \\& \leq \prb{X \leq n^{-\beta/2}\phi} + \prb{X-Y \geq 2n^{-\beta}t\phi} \\
    &\leq en^{-\beta/2} + 2 \exp\left(- \tfrac{t-1}{2}\right).
    \end{align*}

    Now we turn to the numerator. Assume, without loss of generality, that $\|P\|_{2} \le 1$, and let $Z =[PV^{-1}\bm b]^*(V^*V-I)D[PV^{-1}\bm b]$ equal the difference in numerators. Its expectation is given by
    \[\mathbb{E}[Z]  = \mathrm{trace}(D PW^*WP^*(V^*V-I)) = \sum_{j=1}^n D_j P_j \sum_{k \in [n] \backslash j } \bm w_j^*\bm w_k \, \overline{P_k} \, \bm v_k^*\bm v_j.\]
    Using Inequality \ref{ineq:beta} and the property $|P_j|\le |P_k|$ for $j < k$, $|\mathbb{E}[Z]|$ can be bounded as follows:
    \[\big| \mathbb{E}[Z] \big|\leq \sum_{j=1}^n \sum_{k \in [n] \backslash j }  |P_j||P_k| |\bm w_j^*\bm w_k| \, | \bm v_k^*\bm v_j| \leq 2\sum_{j=1}^n |P_j|^2 \sum_{k=1}^{j-1}  |\bm w_j^*\bm w_k| \, | \bm v_k^*\bm v_j| \leq 2n^{-\beta}\phi.\]
    Bounding the variance is slightly more difficult, since $N := [PV^{-1}]^*(V^*V-I)D[PV^{-1}]$ is not necessarily normal. Instead, we bound  $\|N+N^*\|_F^2$ and $\|N-N^*\|_F^2$ separately and make use of Corollary \ref{cor:nonnormalsubexp}. The procedure is very similar to that of the denominator, save for the term $\max_{j,k} |D_j + \overline{D_k}| \leq 2$. We sketch the estimate for $\|N+N^*\|_F^2$ ($\|N-N^*\|_F^2$ is nearly identical):
    \begin{align*} \|N+N^*\|_F^2 &= \mathrm{trace}\left(([D^*(V^*V-I) + (V^*V-I)D]PW^*WP^*)^2\right) \\
    &= \sum_{j=1}^n \sum_{k=1}^n \Bigg(\sum_{\ell \in [n] \backslash j} (\overline{D_j} + D_\ell)\bm v_j^*\bm v_\ell \, P_\ell  \overline{P_k} \, \bm w_\ell^*\bm w_k \Bigg) \Bigg(\sum_{m \in [n] \backslash k}  (\overline{D_k} + D_m) \bm v_k^*\bm v_m \, P_m \overline{P_j} \, \bm w_m^*\bm w_j  \Bigg)\\
    &\le 4 \sum_{j=1}^n \sum_{k=1}^n \Bigg(\sum_{\ell \in [n] \backslash j} |P_\ell|  |\overline{P_k}| \, |\bm v_j^*\bm v_\ell| \,  | \bm w_\ell^*\bm w_k| \Bigg) \Bigg(\sum_{m \in [n] \backslash k} |P_m| |\overline{P_j}| \,  |\bm v_k^*\bm v_m| \,  |\bm w_m^*\bm w_j|  \Bigg) \\
    &\le 4 \sum_{j=1}^n \sum_{k=1}^n |P_j|^2 |P_k|^2   \Bigg(\sum_{\ell =1}^{j-1} |\bm v_j^*\bm v_\ell| \big( |\bm w_\ell^*\bm w_k | + |\bm w_j^*\bm w_k | \big) 
\Bigg) \Bigg(\sum_{m =1}^{k-1}  |\bm v_k^*\bm v_m| \big(  | \bm w_m^*\bm w_j| +  | \bm w_k^*\bm w_j|\big) \Bigg)  \\
    &\le 4n^{-2\beta} \sum_{j=1}^n \sum_{k=1}^n |P_j|^2 |P_k|^2 \|\bm{w}_k\|^2 \| \bm{w}_j \|^2 \\
    &\leq 4n^{-2\beta}\phi^2.
    \end{align*}
Hence, $\mathrm{Var}[Z] \leq 2n^{-2\beta}\phi^2$ and, by Corollary \ref{cor:nonnormalsubexp}, 
\[\prb{|Z| \geq 2n^{-\beta}t\phi} \le  \prb{|Z - \mathbb{E}[Z]| \geq 2n^{-\beta}(t-1)\phi} \leq 4 \exp \left(- \tfrac{t-1}{4} \right).\]

To complete the proof, we union bound the cases that the denominator is too small and the numerator too big. Let $Z_{x} = [PV^{-1}\bm b]^*D[PV^{-1}\bm b]$ and $Z_{y} = [PV^{-1}\bm b]^*V^*VD[PV^{-1}\bm b]$. Setting $t =1+ \ln n$ (note that $2n^{-\beta/2}t<1$ for $\beta >2 \ln(2\ln(e n))/\ln(n)$), we have
\begin{align*}
    \prb{ \frac{|Z_y|}{Y} \ge \frac{1}{1-2n^{-\beta/2}t} \left( \frac{|Z_x|}{X}   + \frac{2t}{n^{\beta/2}} \right)} &\le \prb{Y \le (1-2n^{-\beta/2}t)X} + \prb{|Z_x| - |Z_y| \geq 2n^{-\beta/2}t X} \\
    &\le \prb{X \leq n^{-\beta/2}\phi} + \prb{X-Y \geq 2n^{-\beta}t\phi} + \prb{|Z| \geq 2n^{-\beta}t\phi} \\
    &\le e n^{-\beta/2} + 2 \exp\left(- \tfrac{t-1}{2}\right) + 4 \exp \left(- \tfrac{t-1}{4} \right) \\
    &= e n^{-\beta/2} + 2 n^{-1/2} + 4 n^{-1/4}.
\end{align*}

\end{proof}

The above lemma immediately implies an upper bound for estimating extremal eigenvalues.

\begin{corollary}\label{cor:errsmall}
    Let $n \ge 2$ and $\beta >2 \ln(2\ln(e n))/\ln(n)$, $A \in \mathbb{C}^{n \times n}$ be diagonalizable with eigendecomposition $A = V \Lambda V^{-1}$, where $V$ is $\beta$-normal (see Definition \ref{def:betanormal}), and $\Lambda_k \in \Lambda(A)$ be an extreme point of $\mathrm{conv}(\Lambda(A))$. Then, for $\bm b \sim \mathrm{Unif}(\mathbb{S}^{n-1})$,
    \[\prb{\frac{d(\Lambda_k, W(H_{6m+1}))}{\mathrm{diam}(\mathrm{conv}(\Lambda(A)))} \geq \frac{\frac{6}{m} \ln  \Big( \frac{em \|V^{-1}\bm b\|^2_2 }{6 |[ V^{-1}\bm b]_k|^2}\Big)+  \frac{2\ln(e n)}{n^{\beta/2}}}{1 - \frac{2\ln(e n)}{n^{\beta/2}}}  }\leq e n^{-\beta/2} + 2 n^{-1/2} + 4 n^{-1/4} .\]
\end{corollary}

\begin{proof}
By the proof of Lemma \ref{lm:single_eig}, there exists a polynomial $p \in \mathcal{P}_{6m}$ such that 
   \[\bigg|\Lambda_k - \frac{[p(\Lambda)V^{-1}\bm b]^*\Lambda[p(\Lambda)V^{-1}\bm b]}{[p(\Lambda)V^{-1}\bm b]^*[p(\Lambda)V^{-1}\bm b]}\bigg| \leq \frac{6}{m} \ln  \bigg( \frac{em \|V^{-1}\bm b\|^2_2 }{6 |[ V^{-1}\bm b]_k|^2}\bigg) \, \mathrm{diam}(\mathrm{conv}(\Lambda(A)).\]
   Applying Lemma \ref{lem:VVSmall} with $D = (\Lambda_k I - \Lambda)$ and $P = p(\Lambda)$ completes the proof.
\end{proof}

Combining the above estimate for a single extremal eigenvalue with Proposition \ref{prop:samplepolytope}, we can produce an estimate on the one-sided Hausdorff distance between $\mathrm{conv}(\Lambda(A))$ and $W(H_m)$.

\begin{theorem}\label{thm:eigenhull}
    Let $n \ge 2$ and $\beta >2 \ln(2\ln(e n))/\ln(n)$, $A \in \mathbb{C}^{n \times n}$ be diagonalizable with eigendecomposition $A = V \Lambda V^{-1}$, where $V$ is $\beta$-normal (see Definition \ref{def:betanormal}). Then, for $\bm b \sim \mathrm{Unif}(\mathbb{S}^{n-1})$ and any $0 < \gamma < \min\{\beta/2,1/4\}$,
the one-sided Hausdorff distance between $\mathrm{conv}(\Lambda(A))$ and $W(H_{6m+1})$ is bounded by
\[\frac{\tilde{d}_H(\mathrm{conv}(\Lambda(A)), W(H_{6m+1}))}{\mathrm{diam}(\mathrm{conv}(\Lambda(A)))} \leq\frac{\frac{6}{m} \ln (n^{2+\alpha} \kappa^2(V))+  \frac{2\ln(e n)}{n^{\beta/2}}}{1 - \frac{2\ln(e n)}{n^{\beta/2}}} + \frac{\pi}{n^\gamma}\tan \frac{2 \pi}{n^\gamma}\]
with probability at least
    \[1  - \frac{e^2 m}{6 n^{\alpha}} -  \frac{e}{n^{\beta/2- \gamma}} - \frac{2}{n^{1/2-\gamma}} - \frac{4}{n^{1/4-\gamma}} .\]
\end{theorem}
\begin{proof}
    Let $\mathcal{U} = \mathrm{conv}(\Lambda(A))$, and $L = \pi\,\mathrm{diam}(\mathcal{U})$. By Proposition \ref{prop:samplepolytope}, we can pick $n^\gamma/2$ vertices $S$ to create an interior polytope $\mathcal{V} = \mathrm{conv}(S)$ approximating the convex hull with error at most $\frac{L}{n^\gamma}\tan{\frac{2\pi}{n^\gamma}}$. The optimal polytope $\mathcal{V}$ must have each element of $S$ on the boundary of $\mathcal{U}$. For each $s \in S$, let $\lambda_s,\mu_s \in \Lambda(A)$ be the endpoints of the boundary segment containing $s$. Let $W = \bigcup_{s \in S} \{\lambda_s,\mu_s\}$ and $\mathcal{W} = \mathrm{conv}(W)$. Since $\mathcal{V} \subseteq \mathcal{W} \subseteq \mathcal{U}$, $\mathcal{W}$ approximates $\mathcal{U}$ at least as well as $\mathcal{V}$. 

    By Corollary \ref{cor:errsmall} applied to the eigenvalues in $W$ and Proposition \ref{prop:polytopevertices},
 \[\prb{\frac{\tilde d_H(\mathcal{U}, W(H_{6m+1}))}{\mathrm{diam}(\mathrm{conv}(\Lambda(A)))} \geq \frac{\frac{6}{m} \ln  \Big( \frac{em \|V^{-1}\bm b\|^2_2}{6 \min_k |[ V^{-1}\bm b]_k|^2} \Big)+  \frac{2\ln(e n)}{n^{\beta/2}}}{1 - \frac{2\ln(e n)}{n^{\beta/2}}} + \frac{\pi}{n^\gamma}\tan \frac{2 \pi}{n^\gamma} }\leq \frac{e}{n^{\beta/2- \gamma}} + \frac{2}{n^{1/2-\gamma}} + \frac{4}{n^{1/4-\gamma}} .\]
By Proposition \ref{prop:perturbed_basis} with $M = V^{-1}$ and $t = e m /(6 n^{\alpha})$, we have
\[\Prob\left[ \frac{\min_{k \in [n]} |[V^{-1} \bm b]_k|^2}{\|V^{-1} \bm b\|_2} \le \frac{e m }{6 n^{2+\alpha} \kappa^2(V)} \right] \le \frac{e^2 m}{6 n^{\alpha}} ,\]
completing the proof.
\end{proof}

\section*{Acknowledgements}
The authors would like to thank Louisa Thomas for improving the style of presentation.

{ \small 
	\bibliographystyle{plain}
	\bibliography{main.bib} }

\newpage
\appendix
\section{A Remez-Type Polynomial for the Half Annulus}

Here we provide a proof of Proposition \ref{prop:remez}. Recall that
\[ D_\delta = \{ z \in \mathbb{C} \, | \, \delta \le |z|\le 1, \mathrm{Re}(z) \le 0 \} \]
and
\[R_\epsilon = \{z \in \mathbb{C} \,\big|\, |z| \leq 1, \arg{z} \in [\epsilon, \pi - \epsilon] \cup [\pi + \epsilon, 2 \pi - \epsilon]\} \cup \{z \in \mathbb{C} \,\big|\, |z| \leq 1-\epsilon/8 \}.\]

Given Proposition \ref{prop:remez_disk}, in order to prove Proposition \ref{prop:remez}, it suffices to prove that $P(0) =1$ and $P(D_\delta) \subset R_{\tfrac{2}{3} \delta}$ for $P(z) = (1-\tfrac{\delta}{4}) z^2 + (1-\tfrac{\delta}{8})z + 1$. Clearly, the former is true. The latter is proved below.
        
\begin{proposition}
Let $0<\delta<1$ and $P(z) = (1-\tfrac{\delta}{4}) z^2 + (1-\tfrac{\delta}{8})z + 1$. Then $P(D_\delta) \subset R_{\tfrac{2}{3} \delta}$.
\end{proposition}

\begin{proof}
By the open mapping theorem, it suffices to show that $P(\partial D_\delta) \subset R_{\tfrac{2}{3} \delta}$. We do so by considering the following three portions of $\partial D_\delta$ separately:
\begin{enumerate}
\item $z = e^{i \theta}$, $\theta \in [\tfrac{\pi}{2},\tfrac{3 \pi}{2}]$,
\item $z = c i$, $\delta \le |c| \le 1$,
\item $z = \delta e^{i \theta}$, $\theta \in [\tfrac{\pi}{2},\tfrac{3 \pi}{2}]$.
\end{enumerate}
First, consider $z = e^{i \theta}$, $\theta \in [\tfrac{\pi}{2},\tfrac{3 \pi}{2}]$. We have
\begin{align*}
\left| P(e^{i \theta})\right|^2 &= \left| (1-\tfrac{\delta}{4}) e^{2i \theta} + (1-\tfrac{\delta}{8})e^{i \theta} + 1 \right|^2 \\
&= 1 - \tfrac{\delta}{4} + \tfrac{5 \delta^2}{64} + 4 (1- \tfrac{\delta}{8})^2 \cos \theta + 4 (1 - \tfrac{\delta}{4}) \cos^2 \theta.
\end{align*}
$\left| P(e^{i \theta})\right|^2$ is a convex function of $\cos \theta$ for $\theta \in [\tfrac{\pi}{2},\pi]$, and so it achieves its maximum at one of the endpoints. Therefore,
\begin{align*}
\left| P(e^{i \theta})\right|^2 &\le  1 - \tfrac{\delta}{4} + \tfrac{5 \delta^2}{64} + 4 \max \{(1 - \tfrac{\delta}{4}) - (1- \tfrac{\delta}{8})^2, 0 \} \\
&\le  1 - \tfrac{\delta}{4} + \tfrac{5 \delta^2}{64}. \\
\end{align*}
The right hand side is bounded above by $(1- \tfrac{1}{12} \delta)^2$ for $\delta \in [0,1]$, and so $P(e^{\theta i}) \subset R_{\tfrac{2}{3} \delta}$ for $\theta \in [\tfrac{\pi}{2},\tfrac{3 \pi}{2}]$. Next, consider $z = c i$ for $\delta \le |c| \le 1$. We have 
\[P(ci) = 1 - (1- \tfrac{\delta}{4})c^2 + (1- \tfrac{\delta}{8}) c i.\]
Note that the real part of $P(ci)$ is positive. Without loss of generality, let $c >0$. It suffices to show that
\[ \arg(P(ci)) = \arctan \frac{(1- \tfrac{\delta}{8}) c}{1-(1- \tfrac{\delta}{4})c^2} \ge \tfrac{2}{3} \delta,\]
or, equivalently, that
\[\tan( \tfrac{2}{3} \delta) (1- \tfrac{\delta}{4})c^2 + (1- \tfrac{\delta}{8}) c - \tan( \tfrac{2}{3} \delta) \ge 0.\]
The left hand side is increasing with respect to $c$ for $c>0$, and so
\[\tan( \tfrac{2}{3} \delta) (1- \tfrac{\delta}{4})c^2 + (1- \tfrac{\delta}{8}) c - \tan( \tfrac{2}{3} \delta) \ge \tan( \tfrac{2}{3} \delta) (1- \tfrac{\delta}{4})\delta^2 + (1- \tfrac{\delta}{8}) \delta - \tan( \tfrac{2}{3} \delta).\]
The derivative of the right hand side is bounded below by $\tfrac{3}{10}$ on $[0,1]$, and so its minimum is achieved at $\delta = 0$ and given by zero. Therefore, $P(ci) \subset R_{\tfrac{2}{3} \delta}$ for $\delta \le |c| \le 1$. Finally, consider $z = \delta e^{i \theta}$ for $\theta \in [\tfrac{\pi}{2},\tfrac{3 \pi}{2}]$. We have
\begin{align*}
\left| P(\delta e^{i \theta})\right|^2 &= \left| (1-\tfrac{\delta}{4}) \delta^2 e^{2i \theta} + (1-\tfrac{\delta}{8}) \delta e^{i \theta} + 1 \right|^2 \\
&= 1 - \delta^2 + \tfrac{1}{4} \delta^3 + \tfrac{65}{64} \delta^4 - \tfrac{1}{2} \delta^5 + \tfrac{1}{16} \delta^6 + \delta (2 - \tfrac{1}{4}\delta + 2 \delta^2 - \tfrac{3}{4} \delta^3 + \tfrac{1}{16} \delta^4) \cos \theta +\delta^2(2-\tfrac{1}{2}\delta) \cos^2 \theta
\end{align*}
and
\[\arg(P(\delta e^{i \theta})) = \arctan \frac{(1-\tfrac{\delta}{8}) \delta \sin(\theta) + (1-\tfrac{\delta}{4}) \delta^2 \sin(2\theta) }{1 + (1-\tfrac{\delta}{8}) \delta \cos(\theta) + (1-\tfrac{\delta}{4}) \delta^2 \cos(2\theta)   }.\]
We break our analysis of $z = \delta e^{i \theta}$ for $\theta \in [\tfrac{\pi}{2},\tfrac{3 \pi}{2}]$ into three cases, depending on the size of $\delta$ and value of $\theta$. Let us first consider the size of $\left| P(\delta e^{i \theta})\right|^2$. This quantity is a convex function of $\cos \theta$ for $\theta \in [\tfrac{\pi}{2},\pi]$, and so it achieves its maximum at one of the endpoints. Therefore,
\begin{align*}
\left| P(\delta e^{i \theta})\right|^2 &\le  1 - \delta^2 + \tfrac{1}{4} \delta^3 + \tfrac{65}{64} \delta^4 - \tfrac{1}{2} \delta^5 + \tfrac{1}{16} \delta^6 + \delta \max \{ -(2 - \tfrac{9}{4} \delta + \tfrac{5}{2} \delta^2 - \tfrac{3}{4} \delta^3 + \tfrac{1}{16} \delta^4),0\}\\
&= 1 - \delta^2 + \tfrac{1}{4} \delta^3 + \tfrac{65}{64} \delta^4 - \tfrac{1}{2} \delta^5 + \tfrac{1}{16} \delta^6.
\end{align*}
This quantity is strictly less than $(1 - \tfrac{1}{12} \delta)^2$ for $\delta \in [\tfrac{9}{50},1]$, so we may now restrict our attention to $\delta < \tfrac{9}{50}$. If $\cos \theta \le -\tfrac{1}{12}$, then
\begin{align*}
\left| P(\delta e^{i \theta})\right|^2 &\le  1 - \delta^2 + \tfrac{1}{4} \delta^3 + \tfrac{65}{64} \delta^4 - \tfrac{1}{2} \delta^5 + \tfrac{1}{16} \delta^6 \\ &\quad + \delta \max \{ -(2 - \tfrac{9}{4} \delta + \tfrac{5}{2} \delta^2 - \tfrac{3}{4} \delta^3 + \tfrac{1}{16} \delta^4), -\tfrac{1}{12}(2 - \tfrac{1}{4}\delta + 2 \delta^2 - \tfrac{3}{4} \delta^3 + \tfrac{1}{16} \delta^4) + \tfrac{1}{144} \delta (2 - \tfrac{1}{2} \delta)\}\\
&= 1 - \tfrac{1}{6} \delta  - \tfrac{139}{144} \delta^2 + \tfrac{23}{288} \delta^3 + \tfrac{69}{64} \delta^4 - \tfrac{97}{192} \delta^5 + \tfrac{1}{16} \delta^6,
\end{align*}
and the right hand side is bounded above by $(1- \tfrac{1}{12} \delta)^2$. All that remains is to consider the case where $\delta<\tfrac{9}{50}$ and $ - \tfrac{1}{12}< \cos \theta \le 0$. When $\delta$ and $\cos \theta$ are this small in magnitude, 
\[\mathrm{Re}(P(\delta e^{i \theta})) = 1 + (1-\tfrac{\delta}{8}) \delta \cos(\theta) + (1-\tfrac{\delta}{4}) \delta^2 \cos(2\theta) >0 .\]
Suppose, without loss of generality, that $\theta \in [ \tfrac{\pi}{2},\pi]$. We aim to show that, in this regime, $\arg(P(\delta e^{i \theta})) \ge  \tfrac{2}{3} \delta$, or, equivalently, that
\[(1-\tfrac{\delta}{8}) \delta \sin(\theta) + (1-\tfrac{\delta}{4}) \delta^2 \sin(2\theta) \ge \left(1 + (1-\tfrac{\delta}{8}) \delta \cos(\theta) + (1-\tfrac{\delta}{4}) \delta^2 \cos(2\theta) \right) \tan (\tfrac{2}{3} \delta).\]
We note that $\cos \theta$ and $\cos(2 \theta)$ are negative, $\sin \theta >\tfrac{99}{100}$, and $\sin (2 \theta)> - \tfrac{1}{5}$ for $\theta \in[\tfrac{\pi}{2},\arccos (-\tfrac{1}{12})]$. Combining these bounds with the inequality $\tan(x) \le x + \tfrac{1}{2} x^3$ for $0\le x\le 1/2$, it suffices to show that
\[\tfrac{99}{100}(1-\tfrac{\delta}{8}) \delta  - \tfrac{1}{5} (1-\tfrac{\delta}{4}) \delta^2  -(\tfrac{2}{3} \delta + \tfrac{4}{27} \delta^3) = \delta (\tfrac{97}{300} - \tfrac{259}{800} \delta - \tfrac{53}{540} \delta^2) \ge 0.\]
By inspection, this is indeed the case for $\delta<\tfrac{9}{50}$. Therefore,  $P(\delta e^{\theta i}) \subset R_{\tfrac{2}{3} \delta}$ for $\theta \in [\tfrac{\pi}{2},\tfrac{3 \pi}{2}]$. This completes the analysis of the third and final portion of $\partial D_\delta$, and so $P(\partial D_\delta) \subset R_{\tfrac{2}{3} \delta}$, completing the proof.
\end{proof}

\end{document}

%% file: figures/regions.tikz
\begin{center}
    \includegraphics[scale=0.4]{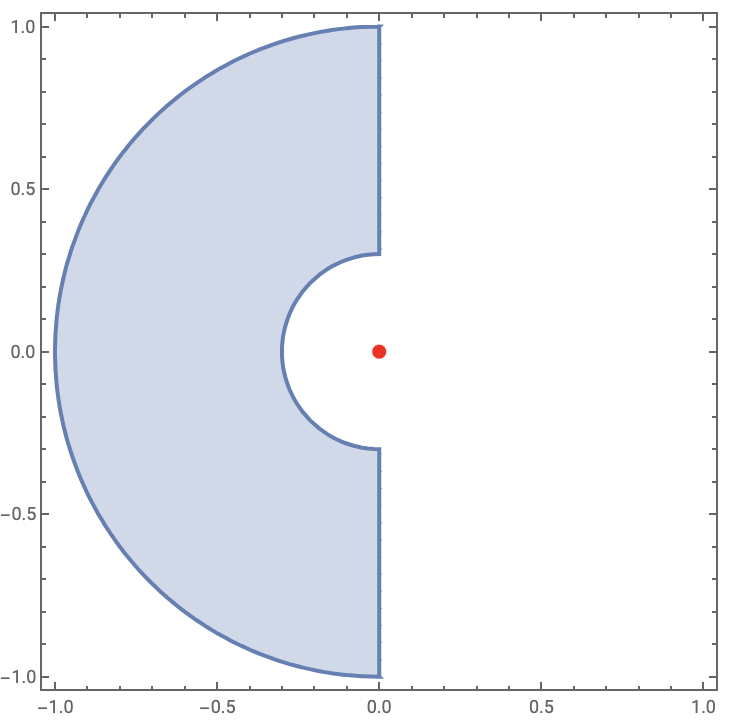}
    \begin{tikzpicture}
        \draw[color=white] (-0.3, 0) circle (0.1);
        \draw[color=white] (1.3, 0) circle (0.1);

        \draw (0.5, 2.5) node[above] {\small{$z \mapsto (1-\frac{\delta}{4})z^2 + (1-\frac{\delta}{8})z + 1$}};
        \draw[->, thick] (0,2.3) -- (1.1,2.3);
    \end{tikzpicture}
    \includegraphics[scale=0.4]{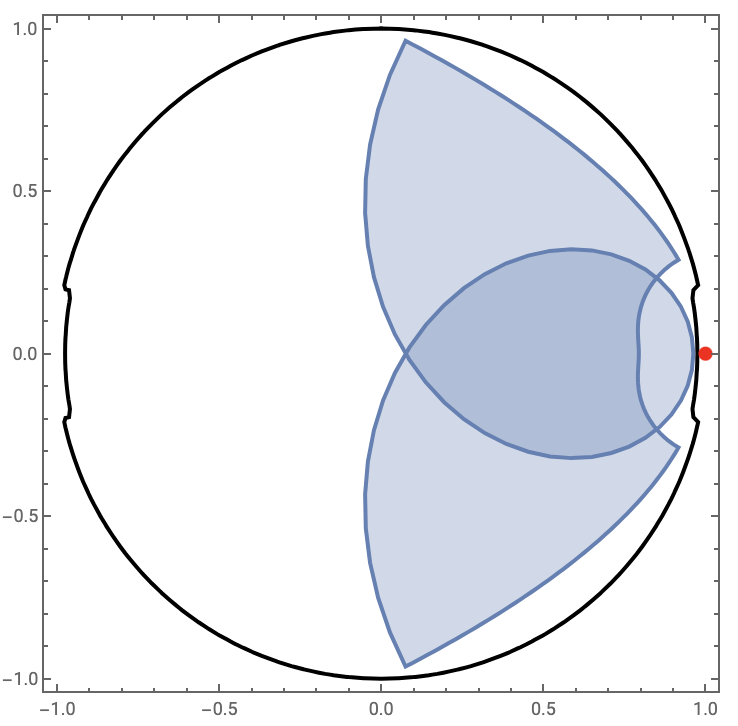}
\end{center}